\documentclass[10pt,a4paper,openany,reqno]{amsart}

\usepackage{amsmath,amsfonts,ams}
\usepackage{latexsym}
\usepackage{amssymb}
\usepackage{color}
\usepackage{epsfig}
\usepackage{graphicx}
\usepackage{subfigure}
\usepackage{mathrsfs}
\usepackage{amscd}
\usepackage{amsthm}
\usepackage{color}
\usepackage{cite}
\usepackage{epstopdf}

\newcommand{\norm}[1]{\| #1\|}

\numberwithin{equation}{section}
\newtheorem{proposition}{Proposition}[section]
\newtheorem{definition}{Definition}[section]
\newtheorem{lemma}{Lemma}[section]
\newtheorem{theorem}{Theorem}[section]

\newtheorem{remark}{Remark}[section]

\let\pa=\partial
\let\al=\alpha

\def\C{\mathop{\bf C\kern 0pt}\nolimits}
\def\DD{\mathop{\bf D\kern 0pt}\nolimits}
\def\K{\mathop{\bf K\kern 0pt}\nolimits}
\def\N{\mathop{\bf N\kern 0pt}\nolimits}
\def\Q{\mathop{\bf Q\kern 0pt}\nolimits}

\newcommand{\la}{\lambda}

\newcommand{\beq}{\begin{equation}}
\newcommand{\eeq}{\end{equation}}
\newcommand{\ben}{\begin{eqnarray}}
\newcommand{\een}{\end{eqnarray}}
\newcommand{\beno}{\begin{eqnarray*}}
\newcommand{\eeno}{\end{eqnarray*}}

 \renewcommand{\theequation}{\arabic{section}.\arabic{equation}}

\makeatletter
\newcommand{\Extend}[5]{\ext@arrow0099{\arrowfill@#1#2#3}{#4}{#5}}
\makeatother

\begin{document}

\title[Defocusing NLS in  dimension two]{Scattering theory   below energy space for   two dimensional nonlinear Schr\"odinger equation}

\author{Changxing Miao}
\address{ Institute of Applied Physics and Computational Mathematics,
        P.O. Box 8009, Beijing 100088, P.R. China.\\
        and Beijing Center for Mathematics and Information
Interdisciplinary Sciences, Beijing, 100048, P.R.China.}

\email{miao\_{}changxing@iapcm.ac.cn}

\author{Jiqiang Zheng}
\address{The Graduate School of China Academy of Engineering Physics, P. O. Box 2101, Beijing, China, 100088}
\email{zhengjiqiang@gmail.com} \maketitle
\begin{abstract}
The purpose of this paper is to illustrate the I-method by studying
low-regularity solutions of the nonlinear Schr\"odinger equation in
two space dimensions. By applying this method, together with the
interaction Morawetz estimate \cite{CGT09,PV},
we establish global well-posedness
and scattering for low-regularity solutions of the equation $iu_t +
\Delta u = \lambda _1|u|^{p_1} u + \lambda _2|u|^{p_2} u$
under certain
assumptions on parameters. This is the first result of this type for
an equation which is not scale-invariant.
In the first step, we establish global well-posedness and scattering for
low regularity solutions of the equation $iu_t + \Delta u = |u|^p u$,
for a suitable range of the exponent $p$ extending the result of Colliander,
Grillakis and Tzirakis [Comm. Pure Appl. Math. 62(2009), 920-968.]
\end{abstract}


\begin{center}
 \begin{minipage}{120mm}
 \end{minipage}
 \end{center}
\section{Introduction}

\noindent We consider the initial-value problem for the defocusing
nonlinear Schr\"odinger equation
\begin{align} \label{equ1.1}
\begin{cases}    (i\partial_t+\Delta)u= f(u),\quad
(t,x)\in\R\times\R^2,
\\
u(0,x)=u_0(x),
\end{cases}
\end{align}
where $u:\R_t\times\R_x^2\to \mathbb C$. If $f(u)=|u|^pu$, the equation in   \eqref{equ1.1}
is invariant under the scaling transform
\begin{equation}\label{scale}
u(t,x)\mapsto \lambda^{2/p}u(\lambda^2t, \lambda x),\quad\text{for any }\; \lambda>0,
\end{equation}
and this  scaling property leads to the notion of \emph{criticality} for
problem \eqref{equ1.1}.  Indeed, one can verify that the homogeneous Sobolev space $\dot{H}^{s_c}(\mathbb R^2)$ with the critical regularity index
$s_c:=1-\frac2p$
  is invariant under scaling
\eqref{scale}.
 Then, for  every  $u_0\in H_x^s(\R^2)$,  we refer to  the problem \eqref{equ1.1} as \emph{critical} if  $s=s_c$,
\emph{subcritical} for  $s>s_c$,   and \emph{supercritical} if  $s<s_c$.

If a smooth solution $u$ of problem \eqref{equ1.1}  has sufficient decay at
infinity, it conserves  mass
\begin{equation}
M(u)=\int_{\mathbb{R}^2}|u(t,x)|^2dx=M(u_0)
\end{equation}
and the energy
\begin{equation}
E(u(t))=\int_{\mathbb{R}^2}\Big(\tfrac12|\nabla
u(t)|^2+F(u)\Big)dx=E(u_0),~F(u)=\int_0^{|u|}f(s)ds.
\end{equation}

The global well-posedness and scattering theory  for the defocusing
Schr\"odinger equation(NLS)
\begin{align} \label{equ1.1.123}
\begin{cases}    (i\partial_t+\Delta)u=|u|^pu,\quad
(t,x)\in\R\times\R^d,
\\
u(0,x)=u_0(x)\in H^s(\R^d),
\end{cases}
\end{align}
has been intensively studied in papers
\cite{Bour98,Cav,CaW,CGT07,CKSTT02,CKSTT0301,CKSTT03,GV85,MS,Na}.
Recall that a global solution $u$ to \eqref{equ1.1.123} scatters in
$H^s(\R^d)$, if there exist unique state $u_{\pm}\in H^{s}_x(\R^d)$ such
that
        $$\lim_{t\to\pm\infty}\norm{u(t)-e^{it\Delta}u_{\pm}}_{H_x^{s}(\R^d)}=0.$$
In the energy-subcritical case, i.e.
for $p\in\big(\tfrac4d,\tfrac4{d-2}\big)$ if $d\geq3$  and for
$p\in\big(\tfrac4d,+\infty\big)$ if $d\in\{1,2\}$, for every $u_0\in H^1(\R^d)$, it
is easy to prove  the global well-posedness for problem \eqref{equ1.1.123}
by combining the Strichartz estimate together with a standard fixed point argument
and the conservation of energy. Ginibre and Velo \cite{GV85} proved
the scattering in spatial dimension $d\geq3$ by making use of the
almost finite propagation speed
$$\int_{|x|\geq a} |u(t,x)|^2 \; dx \leq \int \min\left(\frac{|x|}{a},
1\right) |u(t_0)|^2 \; dx + \frac{C}{a} \cdot |t-t_0|$$ for large
spatial scale and the classical  Morawetz inequality in  \cite{LS}
\begin{equation}\label{cmrsz}
\iint_{\R\times \R^d}\frac{|u(t,x)|^{p+2}}{|x|} \; dtdx \lesssim
\|u\|_{L_t^\infty\dot H^\frac12}^2\lesssim C\big(M(u_0), E(u_0)\big)
\end{equation} for small spatial scale. It is well known that the  Morawetz estimate is an essential tool in the proof of scattering
for the nonlinear dispersive equations such as nonlinear Schr\"{o}dinger equations and nonlinear Klein-Gordon equations.
  A classical Morawetz
inequality was first derived by Morawetz \cite{M} for the nonlinear  Klein-Gordon
equation,  and then extended by Lin and Strauss \cite{LS}
 to the nonlinear Schr\"odinger equation with $d\geq3$ in order to obtain the
scattering for slightly more regular solutions. Next, Nakanishi
\cite{Na} extended the above Morawetz inequality to the dimensions
$d\in\{1,2\}$ by considering certain variants of the Morawetz estimate with
space-time weights and consequently he proved the scattering in low dimensions.

The Morawetz estimate \eqref{cmrsz} plays an important role in
the proof of scattering for the problem \eqref{equ1.1.123} in the
energy-subcritical case, but it does not work so powerfully in the
energy-critical case (i.e.  for $p=\tfrac4{d-2}$ if $d\geq3$). Thus,  to
obtain the scattering in the critical case, it is a very difficult
problem. An  essential  breakthrough came from Bourgain
\cite{Bo99a} who exploited the `induction on energy' technique
and the following spatial-localized  Morawetz inequality
\begin{equation}\label{CKSTT08}
\int_I\int_{|x|\leq C|I|^\frac12}\frac{|u(t,x)|^6}{|x|} \; dtdx
\lesssim |I|^\frac12C\big(\|u\|_{L_t^\infty \dot
H^1(I\times\R^3)}\big)
\end{equation}
to obtain the scattering of  radial solutions to problem
\eqref{equ1.1.123} with $p=4$ in the energy space $\dot H^1(\R^3)$. Next,
 Colliander, Keel,
Staffilani, Takaoka and Tao  (I-team) \cite{CKSTT08} removed the
radial symmetry assumption in \cite{Bo99a}, and solved this longtime standing problem
 through the Bourgain `induction on
energy' technique and the frequency localized type of the
interaction Morawetz inequality
\begin{equation}\label{inmwe}
\big\||\nabla|^{\frac{3-d}2}(|u|^2)\big\|_{L_{t,x}^2(\R\times\R^d)}^2\lesssim
\|u_0\|_{L_x^2}^2\|u\|_{L_t^\infty(\R;\dot H^\frac12)}^2,\qquad
d\geq1.
\end{equation}
This interaction Morawetz  inequality was first derived by
I-team in their work \cite{CKSTT04} in spatial dimension $d=3$ and then
extended to $d\geq4$ in \cite{TVZ}.
 Colliander,  Grillakis and  Tzirakis\cite{CGT09}, Planchon and Vega \cite{PV}
independently proved \eqref{inmwe} in dimensions $d\in\{1,2\}.$ As a
byproduct, one can easily give another simpler proof of  the result
of  Ginibre and Velo \cite{GV85}, see \cite{TVZ} for more detail. We
also refer the reader to \cite{GV2010} for the exposition on the
Morawetz inequalities  and their applications.

 The
interaction Morawetz inequality plays also an important role in the study of a low
regularity problem. Where we ask what is the minimal $s$  to ensure that problem \eqref{equ1.1.123} has either  a local solution  or
a  global solution for which the scattering hold? Such a problem was first considered by Cazenave and
Weissler \cite{CaW}, who proved that problem \eqref{equ1.1.123} is
locally well posed  in $H^s(\R^d)$ with $s\geq\max\{0,s_c\}$ and globally well posed together with  scattering for small data in $\dot H^{s_c}(\R^d)$ with $s_c\ge 0$.
 They used Strichartz estimates in the framework of Besov spaces.
On the other hand, since the
lifespan of local solutions depend only on the $H^s$-norm of the
initial data for $s>\max\{0,s_c\}$, one can easily obtain the global
well-posedness for \eqref{equ1.1.123} in two special cases: the mass
subcritical case ($p<\tfrac4d$) for $L^2_x(\R^d)$-initial data and
the energy-subcritical case (for $p<\tfrac4{d-2}$,if $d\geq3$ or
for $p<+\infty$ if $d\in\{1,2\}$) for $H^1_x(\R^d)$-initial data by using  the
conservation of mass and energy respectively.

This leaves  the open  problem on  global well-posedness in
$H^s(\R^d)$ in the intermediate regime $0\leq s_c\leq s<1$. The
first progress on this direction came from the Bourgain `Fourier
truncation method \cite{Bour98} where
refinements of Strichartz' inequality \cite{Bourimrn}, high-low
frequency decompositions and perturbation methods were used to show that
problem \eqref{equ1.1.123} with $p=2$ is globally wellposed in
$H^s(\R^3)$ with $s>\frac{11}{13}$ such that
\begin{equation}
u(t)-e^{it\Delta}u_0\in H^1(\R^3).
\end{equation}
This leads to  the I-method which was derived by
Keel and Tao in the study of wave maps \cite{KT98IMRN}.
Subsequently, I-team  developed the I-method
 to treat many low regularity problems including the
nonlinear Schr\"odinger equations with derivatives, the one dimensional quintic
NLS, and the cubic NLS in two and three
dimensions\cite{CKSTT01,CKSTT02,CKSTT0201,CKSTT0301,CKSTT03,CKSTT07}.
Compared with the result in \cite{Bour98}, I-team
 also obtained the scattering in $H^s(\R^3)$ with
$s>\tfrac56$ by using the I-method and the interaction Morawetz
estimate \eqref{inmwe} in \cite{CKSTT04}.   Dodson \cite{Dod13} extended those results to
$s>\tfrac57$ by means of a linear-nonlinear decomposition, and then Su
\cite{Su}  to $s>\tfrac23$. For the cubic NLS in
dimension two (corresponding to the mass-critical), I-term further
exploited the improved I-method in \cite{CKSTT07} to get the global well-posedness
for $s>\tfrac12$. Colliander, Grillakis and
Tzirakis \cite{CGT07} extended it to $s>\frac25$  by means of the
I-method and the improved interaction Morawetz inequalities.
Laterly, Colliander and Roy \cite{CR08} improved these results to
$s>\tfrac13$.   Subsequently, Dodson \cite{Dod} showed the global
well-posedness for $s>\tfrac14$ by improving the almost Morawetz
estimates from \cite{CGT07}.

The study of a low regularity problem stimulates the development of the
scattering in $L^2(\R^d)$ for the mass-critical problem (i.e. for
$p=\tfrac4d$). Dodson\cite{Dodson3,Dodson2,Dodson1} developed so called  long-time Strichartz estimates
to prove  the
global well-posedness and scattering in $L_x^2$-space   by making use
of a concentration-compactness approach  and  the idea of I-method.

\vskip 0.2cm

Now, let us describe   the I-method, which consists in  smoothing out the $H^s$-initial data with $0<s<1$ in order
to access a good local and global theory available at the $H^1$-regularity. To do it, we define the Fourier multiplier $I$ by
$$\widehat{Iu}(\xi):=m(\xi)\hat{u}(\xi),$$ where $m(\xi)$ is a
smooth radial decreasing cut off function such that
\begin{equation}\label{mxidy}
m(\xi)=\begin{cases} 1, ~~  \qquad\qquad|\xi|\leq N,\\
\Big(\frac{|\xi|}{N}\Big)^{s-1},\quad |\xi|\geq2N.
\end{cases}
\end{equation}
Thus, $I$ is the identity operator on frequencies $|\xi|\leq N$ and
behaves like a fractional integral operator of order $1-s$ on higher
frequencies. It is easy to show  that the operator $I$ maps $H^s$ to
$H^1$. Moreover, we have
\begin{equation}\label{equcont}
 \|u\|_{H^s}\lesssim\|Iu\|_{H^1}\lesssim
N^{1-s}\|u\|_{H^s}.
\end{equation}
Thus, to prove that problem  \eqref{equ1.1.123} is globally well-posed in
$H^s(\R^d)$, it suffices to show that $E(Iu(t))<+\infty$ for all
$t\in\R$. Since $Iu$ is not a solution to \eqref{equ1.1.123},
the modified energy $E(Iu)(t)$ is not
conserved.
Indeed, we have
\begin{equation}\label{mdfeninc}
\frac{d}{dt}E\big(Iu(t)\big)={\rm
Re}\int_{\R^d}\overline{Iu_t}\Big[|Iu|^pIu-I(|u|^pu)\Big]dx.\end{equation}
Thus,  the key idea is to show that the
modified energy $E(Iu)$ is an `almost conserved' quantity in the
sense that its derivative $\tfrac{d}{dt}E\big(Iu(t)\big)$ will decay
with respect to a large parameter $N$. This will allow us to control
$E(Iu)$ on time interval where the local solution exists
and we can iterate this estimate to obtain a global in time control  of the solution
by means of  the bootstrap argument, see Section 3 for more details.
Then immediately we get a bound for the $H^1$-norm of $Iu$ which
will give us an $H^s$-bound for the solution $u$ by inequality
\eqref{equcont}.

To deal with equality \eqref{mdfeninc}, one needs complicated estimates
on the commutator $I(|u|^pu)-|Iu|^pIu$. When $p$ is
an even integer, one can write the commutator explicitly by means of the
Fourier transform and to control it by multilinear harmonic analysis,
see \cite{CGT07,CGT09,CHVZ,CKSTT07,CR08,Dod,Dod13,Dod14,Su,Tz} for
considerations of the algebraic nonlinearity $f(u)=|u|^{2k}u$ with
$k\in\mathbb{N}$ in $\mathbb R^d$($d=1,2$) and cubic NLS in  $\mathbb R^3$.
 Colliander, Grillakis and Tzirakis \cite{CGT09}
proved that a solution to \eqref{equ1.1.123} with
$f(u)=|u|^{2k}u$  is global and scatters for $s>1-\tfrac1{4k-3}$ in
 $\mathbb R^2$. Recently, by exploiting the long-time
Strichartz estimate in the Koch-Tataru space $U_\Delta^2$ and
$V_\Delta^2$ (see \cite{KT07,KT12}), Dodson \cite{Dod14} extended
this result to $s>1-\tfrac1k$ for radial initial data.

%

\vskip 0.2cm

Unfortunately, the above method for estimating \eqref{mdfeninc}
depends heavily on the exact form of the nonlinearity. Therefore,
this method fails when $p$ is not an even integer. For arbitrary
$p\in(0,4/(d-2))$ and $d\geq3$, by relying on more rudimentary tools
as Taylor's expansion and Strichartz estimates, I-team
\cite{CKSTT03} obtained polynomial growth of the $H^s$-norm of
solutions, and so the global well-posedness for problem \eqref{equ1.1.123}
with $s$ sufficiently close to $1$. However, their bounds are
insufficient to yield scattering. Subsequently, Visan and Zhang
\cite{VZ} combined the I-method and the a priori interaction
Morawetz estimate \eqref{inmwe} to show that scattering holds in
$H^s(\R^d)(d\geq3)$ for $s$ being larger than some
$s_0(d,p)\in(0,1)$. This method is weaker than the multilinear
multiplier method when $p$ is an even integer.

\vskip0.2cm I-method also relies on the scale-invariance
of the equation in \eqref{equ1.1.123}. Therefore, adding a
perturbation to the equation which destroys the scale invariance, is
of particular interest. By this reason, we  study
the nonlinear Schr\"{o}dinger equation \eqref{equ1.1.123}
 which is perturbed by a  lower-order
nonlinearity
\begin{align} \label{equ1.1.1c}
\begin{cases}    (i\partial_t+\Delta)u=|u|^{p_1}u+|u|^{p_2}u,\quad
(t,x)\in\R\times\R^2,~p_1<p_2,
\\
u(0,x)=u_0(x).
\end{cases}
\end{align}
 We look for answers to the
following questions: under which conditions on
 $p_1$ and $p_2$ a solution to problem \eqref{equ1.1.1c} is unique global
 in time  in $H^s(\R^2)$ with suitable $s$, and is scattering?  We use a certain perturbative and scale technique.
We first remove the term $|u|^{p_2}u$ and study the global
well-posedness and scattering for \eqref{equ1.1} with general
nonlinearity  $f(u)=|u|^pu$
by arguments of \cite{VZ} combined with the a priori
interaction Morawetz estimates in \cite{CGT09,PV}. Then, we apply the
I-method to  an  equation derived from that in \eqref{equ1.1.1c} by the  scaling transform \eqref{scale}.

Now, we collect our results into the following theorems. We define
\begin{equation}
s_0:=\max\big\{\tfrac{1+s_c}2,~\tfrac{p}{p+1},\quad s_1\big\},
s_c=1-\tfrac2p,
\end{equation}
and $s_1$ is the positive root of the quadratic equation
$$s^2+2s_cs+s_c^2-4s_c=0.$$

\begin{theorem}\label{theorem}
Assume that $u_0\in H^s(\R^2)$ with $s\in(s_0,1)$,
$p\geq\tfrac{11}{4}$. Then the solution $u$ to \eqref{equ1.1}  with
$f(u)=|u|^pu$ is global and scatters in the sense that there exist
unique $u_{\pm}\in H^{s}_x(\R^2)$ such that
\begin{equation}\label{scatters}
\lim_{t\to\pm\infty}\norm{u(t)-e^{it\Delta}u_{\pm}}_{H_x^{s}(\R^2)}=0.
\end{equation}
\end{theorem}
\begin{remark}
There exists a gap for the region $2<p<\tfrac{11}{4}$. The restriction to $p\geq\tfrac{11}{4}$ comes from estimate
\eqref{equ3.1d} in Proposition \ref{zitinc} and estimate \eqref{restricti} in  Proposition \ref{enerin},  since the
classical interaction Morawetz estimates are  not good enough to
control this time-space norm. We refer to Propositions \ref{zitinc} and
\ref{enerin} for more detail.
\end{remark}

Now, we want to deal with the case  $f(u)=|u|^pu$
with $p<\tfrac{11}{4}$. We will apply the following improved
interaction Morawetz estimates in  \cite{CGT07}
\begin{equation}\label{improved}
\int_0^T\int_{\R^2}|u(t,x)|^4dxdt\lesssim
T^\frac13\|u_0\|_{L_x^2}^2\|u\|_{L_t^\infty([0,T],\dot
H^{1/2}_x)}^2+T^\frac13\|u_0\|_{L_x^2}^4
\end{equation}
instead of the following classical interaction Morawetz estimates
in  \cite{CGT09,PV}
\begin{equation}\label{equ2.14rest}
\|u\|_{L_{t,x}^5(I\times\R^2)}\lesssim\|u\|_{L_t^\infty(I;\dot
H^\frac12(\R^2))}^\frac{2}{5}\|u_0\|_2^\frac3{5}.
\end{equation}
 The
estimate \eqref{improved} will help us to obtain global
well-posedness  with the lower order $p$. But  the $H^s$-norm of the
solution depends on the polynomial growth of time, which is
insufficient to yield scattering. Let us define
\begin{equation}
\tilde s_0:=\max\big\{\tfrac{1+s_c}2,~\tfrac{p}{p+1},~ \tilde
s_1\big\}, s_c=1-\tfrac2p,
\end{equation}
and $\tilde s_1$ be the positive root of the quadratic equation
$$3(s-s_c)^2-2(1+6s_c)(1-s)=0.$$
\begin{theorem}\label{theorem1.2}
Assume that $u_0\in H^s(\R^2)$ with $s\in(\tilde s_0,1)$ and $p>2$.
Then the solution $u$ to \eqref{equ1.1}  with $f(u)=|u|^pu$ is
global. Furthermore, we have the polynomial growth of the $H^s$-norm
of the solution,
\begin{equation}
\sup_{t\in[0,T]}\|u(t)\|_{H^s(\R^2)}\leq
C\big(\|u_0\|_{H^s(\R^2)}\big)(1+T)^{\frac{1-s}{3(s-s_c)^2-2(1+6s_c)(1-s)}+},\qquad
\forall~T>0.
\end{equation}
\end{theorem}

%
%

\vskip 0.2cm

Now we turn to problem \eqref{equ1.1.1c} with
$p_2=2k,~k\in\mathbb N$ and $p_1=p$. Denote
$$s_c^{(1)}=1-\tfrac2p,~~s_c^{(2)}=1-\tfrac1k,$$
and
$$\tilde s_3:=\max\Big\{\tfrac{1+s_c^{(1)}}2,\tfrac{2k}{2k+1},\tfrac{5s_c^{(2)}}{4s_c^{(2)}+1},~s_3\Big\}$$
 where $s_3$ is the positive root of the quadratic
equation
$$s^2-(s_c^{(1)}+s_c^{(2)}-\al)s-\al=0,~\al=4s_c^{(2)}-\tfrac{9(2-\frac{p}k)}{2(p+2)}.$$
\begin{theorem}\label{thm1.2}
Assume that $u_0\in H^s(\R^2)$ with
$$s\in(\tilde s_3,1),~2k>p\geq\tfrac{11}{4}~~\text{and}~~1<
k\in\mathbb{N}.$$ Then the solution $u$ to \eqref{equ1.1} with
$f(u)=|u|^{2k}u+|u|^pu$ is global and scatters in $H^s(\R^2)$.
\end{theorem}

\begin{remark} A simple computation shows $\tilde s_3>s_0$ here. Our
argument also works in the higher dimensional case. By the same way
as in the proof of Theorem \ref{theorem1.2}, one can also use the improved
interaction Morawetz estimates \eqref{improved} to achieve the
global well-posedness of \eqref{equ1.1.1c} with
$p\in\big[2,\tfrac{11}4\big]$.
\end{remark}

\vskip 0.2cm

Finally,  we give the global well-posedness and scattering result
for \eqref{equ1.1} with more general nonlinearity
$f(u)=|u|^{p_1}u+|u|^{p_2}u$ by the same arguments as those in the proofs of
Theorems \ref{theorem} and \ref{thm1.2}.

\begin{theorem}\label{thm1.3}
Assume that $u_0\in H^s(\R^2)$ with
\begin{equation*}
s\in\big(\max\big\{\tfrac{1+s_c^{(2)}}2,~\tfrac{p_2}{p_2+1},~
s_2\big\},~1\big),
s_c^{(j)}=1-\tfrac2{p_j},~j=1,2,~\tfrac{11}{4}\leq p_1<p_2
\end{equation*}
and $s_2$ is the positive root of the quadratic equation
$$s^2+2s_c^{(2)}s+(s_c^{(2)})^2-4s_c^{(2)}=0.$$
 Then the solution of problem \eqref{equ1.1}  with
$f(u)=|u|^{p_1}u+|u|^{p_2}u$ is global and scatters in $H^s(\R^2)$.
\end{theorem}

The paper is organized as follows. In Section $2$,  as
preliminaries, we gather some notations and recall the Strichartz
estimate for NLS and some nonlinear estimates. In Section $3$, we will
 prove
Theorem \ref{theorem} by making use of  I-method together with the interaction Morawetz estimates.
 In Section $4$, we will utilize I-method
and the improved interaction Morawetz inequalities to show Theorem
\ref{theorem1.2}. We prove Theorem \ref{thm1.2} in Section $5$ based
on Theorem \ref{theorem}. In Appendix, we state a result in one
dimension.




\section{Preliminaries}
\subsection{Notations}
 To simplify our inequalities, we introduce
the symbols $\lesssim, \thicksim, \ll$. If $X, Y$ are nonnegative
quantities, we write either  $X\lesssim Y $ or $X={\mathcal O}(Y)$ to denote the estimate
$X\leq CY$ for some $C$, and $X \thicksim Y$ to denote the estimate
$X\lesssim Y\lesssim X$. We use $X\ll Y$ to mean $X \leq c Y$ for
some small constant $c$. We use $C\gg1$ to denote various large
finite constants, and $0< c \ll 1$ to denote various small
constants. For every $r$ such that  $1\leq r \leq \infty$, we denote by $\|\cdot
\|_{r}$ the norm in the Lebesgue space  $L^{r}=L^{r}(\mathbb{R}^d)$ and by $r'$ the
conjugate exponent defined by $\frac{1}{r} + \frac{1}{r'}=1$. We
denote  by $a\pm$   quantities of the form $a\pm\epsilon$ for any
$\epsilon>0.$ We always assume $d=2$ and $s<1$.

Let $f(z):=|z|^pz$, then
$$f_z(z):=\tfrac{\pa f}{\pa
z}(z)=\tfrac{p+2}2|z|^p~~~~\text{and}~~~~f_{\bar z}(z):=\tfrac{\pa
f}{\pa \bar z}(z)=\tfrac{p}2|z|^p\tfrac{z}{\bar z}.$$
 We denote $F'$ to be the vector $(f_z,f_{\bar z})$ and use the
 notation
$$w\cdot f'(z)=wf_z(z)+\bar wf_{\bar z}(z).$$
In particular, we get by the chain rule
$$\nabla f(u)=\nabla u\cdot f'(u),$$
and
$$|f'(z)-f'(w)|\lesssim|z-w|\big(|z|+|w|\big)^{p-1},~\quad p>1.$$

The Fourier transform on $\mathbb{R}^2$ is defined by
\begin{equation*}
\aligned \widehat{f}(\xi):= \big( 2\pi
\big)^{-1}\int_{\mathbb{R}^2}e^{- ix\cdot \xi}f(x)dx ,
\endaligned
\end{equation*}
giving rise to the fractional differentiation operators
$|\nabla|^{s}$ and $\langle\nabla\rangle^s$ defined  by
\begin{equation*}
\aligned
\widehat{|\nabla|^sf}(\xi):=|\xi|^s\hat{f}(\xi),~~\widehat{\langle\nabla\rangle^sf}(\xi):=\langle\xi\rangle^s\hat{f}(\xi),
\endaligned
\end{equation*} where $\langle\xi\rangle:=1+|\xi|$.
This helps us to define the homogeneous and inhomogeneous Sobolev
norms
\begin{equation*}
\big\|f\big\|_{\dot{H}^s_x(\R^2)}:= \big\|
|\xi|^s\hat{f}\big\|_{L^2_x(\R^2)},~~\big\|f\big\|_{{H}^s_x(\R^2)}:=
\big\| \langle\xi\rangle^s\hat{f}\big\|_{L^2_x(\R^2)}.
\end{equation*}

We will also need the Littlewood-Paley projection operators. Let
$\varphi(\xi)$ be a smooth bump function adapted to the ball
$|\xi|\leq 2$ which equals 1 on the ball $|\xi|\leq 1$. For each
dyadic number $N\in 2^{\mathbb{Z}}$, we define the Littlewood-Paley
operators
\begin{equation*}
\aligned \widehat{P_{\leq N}f}(\xi)& :=
\varphi\big(\tfrac{\xi}{N}\big)\widehat{f}(\xi), \\
\widehat{P_{> N}f}(\xi)& :=
\Big(1-\varphi\big(\tfrac{\xi}{N}\big)\Big)\widehat{f}(\xi), \\
\widehat{P_{N}f}(\xi)& :=
\Big(\varphi\big(\tfrac{\xi}{N}\big)-\varphi\big(\tfrac{2\xi}{N}\big)\Big)\widehat{f}(\xi).
\endaligned
\end{equation*}
Similarly,  we can define $P_{<N}$, $P_{\geq N}$, and $P_{M<\cdot\leq
N}=P_{\leq N}-P_{\leq M}$, whenever $M$ and $N$ are dyadic numbers.
Especially, we denote $P_1:=P_{\leq1}$. We will frequently write
$f_{\leq N}$ for $P_{\leq N}f$ and similarly for the other
operators.

The Littlewood-Paley operators commute with derivative operators,
the free propagator, and the conjugation operation. They are
self-adjoint and bounded on every space  $L^p_x(\mathbb R^2)$ and $\dot{H}^s_x(\mathbb R^2)$
for $1\leq p\leq \infty$ and $s\geq 0$. Moreover, they also obey the
following
 Bernstein estimates.

\begin{lemma}[Bernstein estimates]   For every  $s\geq 0$, $1\leq p\leq q \leq \infty$, and $N\in\mathbb N$,
we have
\begin{eqnarray*}\label{bernstein}
 \big\| P_{\geq N} f \big\|_{L^p(\R^2)} & \lesssim & N^{-s} \big\|
|\nabla|^{s}P_{\geq N} f \big\|_{L^p(\R^2)}, \\
\big\||\nabla|^s P_{\leq N} f \big\|_{L^p(\R^2)} & \lesssim  & N^{s}
\big\|
P_{\leq N} f \big\|_{L^p(\R^2)},  \\
\big\||\nabla|^{\pm s} P_{N} f \big\|_{L^p(\R^2)} & \thicksim &
N^{\pm s} \big\|
P_{N} f \big\|_{L^p(\R^2)},  \\
\big\| P_{\leq N} f \big\|_{L^q(\R^2)} & \lesssim &
N^{\frac{2}{p}-\frac{2}{q}} \big\|
P_{\leq N} f \big\|_{L^p(\R^2)},  \\
\big\| P_{ N} f \big\|_{L^q(\R^2)} & \lesssim &
N^{\frac{2}{p}-\frac{2}{q}} \big\|P_{ N} f \big\|_{L^p(\R^2)}.
\end{eqnarray*}

 \end{lemma}

 \subsection{Strichartz estimates}\label{sze}
 Let $e^{it\Delta}$ be the free Schr\"odinger propagator given by
    \begin{equation}\label{explicit formula}
    [e^{it\Delta}f](x)=\tfrac{1}{4\pi it}\int_{\R^2} e^{i\vert x-y\vert^2/4t}f(y)\,dy, \quad t\neq 0.
    \end{equation}
 Obviously, it satisfies  the
dispersive estimate
    $$\norm{e^{it\Delta}f}_{L_x^\infty(\R^2)}\lesssim\vert t\vert^{-1}\norm{f}_{L_x^1(\R^2)}, \quad t\neq 0.$$
Interpolating  above inequality with
$\norm{e^{it\Delta}f}_{L_x^2(\R^2)}\equiv \norm{f}_{L_x^2(\R^2)}$
then yields
\begin{equation}\label{dispers}
\big\|e^{it\Delta}f \big\|_{L^q_x(\R^2)} \leq C|t|^{-(1-\frac2{q})}
\|f\|_{L^{q'}_x(\R^2)}, \quad t\neq 0
\end{equation}
for  $2\leq q\leq\infty$. This inequality  implies the
classical Strichartz estimates by the standard $TT^*$ argument,
 which we will state below. First, we
need the following definition.

\begin{definition}[Admissible pairs]\label{def1} A pair of exponents $(q,r)$ is called \emph{Schr\"odinger
admissible} in $\R^2$,  which we  denote by $(q,r)\in \Lambda_{0}$ if
$$2\leq
q,r\leq\infty,~\tfrac{1}{q}+\tfrac1r=\tfrac12,~\text{and}~(q,r)\neq(2,\infty).$$
For a spacetime slab $I\times\R^2$, we define the Strichartz norm
    $$\norm{u}_{S^0(I)}:=\sup\big\{\|u\|_{L_t^{q}L_x^{r}(I\times\R^2)}:(q,r)\in\Lambda_0, ~q\geq2+\epsilon_0\big\},$$
where  $0<\epsilon\ll1_0$. We denote $S^0(I)$ to be the closure of
all test functions under this norm.
\end{definition}

We  now state the standard Strichartz estimates in the form that
we will need later.

\begin{proposition}[Strichartz estimates \cite{GV, KeT98, St}]\label{prop1}
Let $s\geq 0$ and suppose $u:I\times\R^2\to\C$ is a solution to
$(i\partial_t+\Delta)u=\sum\limits_{j=1}^{m}F_{j}$. Then
\begin{equation}\label{strest}
\norm{\vert\nabla\vert^s u}_{S^0(I)}\lesssim\norm{\vert\nabla\vert^s u(t_0)}_{L_x^2(\R^2)}+\sum\limits_{j=1}^{m}
    \big\|\nabla\vert^s
    F_j\big\|_{L_t^{q_j'}L_x^{r_j'}(I\times\R^2)}
\end{equation}
for any admissible pairs $(q_j,r_j)$ and $t_0\in I$.
\end{proposition}

\vskip 0.2cm

\subsection{Nonlinear estimate}

For $N>1$, we define the Fourier multiplier $I:=I_N$ given  by
$$\widehat{Iu}(\xi):=m(\xi)\hat{u}(\xi),$$
where $m(\xi)$ is a smooth radial decreasing cut off function by
\eqref{mxidy}. Let us collect basic properties of $I$.

\begin{lemma}[\cite{VZ}]\label{highcont} Let $1<p<\infty$ and $0\leq\sigma\leq
s<1$. Then,
\begin{align}\label{equ2.1}
\|If\|_{L^p}\lesssim&\|f\|_{L^p},\\\label{equ2.2}
\big\||\nabla|^\sigma
P_{>N}f\big\|_{L^p}\lesssim&N^{\sigma-1}\big\|\nabla
If\big\|_{L^p},\\\label{equ2.3}
\|f\|_{H^s}\lesssim\|If\|_{H^1}\lesssim&N^{1-s}\|f\|_{H^s}.
\end{align}
\end{lemma}

We will also need the following fractional calculus estimates.

\begin{lemma}[\cite{CW}]\label{fensc}
$(i)$ $($Fractional product rule$)$ Let $s\geq0$, and
$1<r,r_j,q_j<\infty$ satisfy $\frac1r=\frac1{r_i}+\frac1{q_i}$ for
$i=1,2$. Then
\begin{equation}\label{moser}
\big\||\nabla|^s(fg)\big\|_{L_x^r(\R^2)}\lesssim\|f\|_{{L_x^{r_1}(\R^2)}}\big\||\nabla|^sg
\big\|_{{L_x^{q_1}(\R^2)}}+\big\||\nabla|^sf\big\|_{{L_x^{r_2}(\R^2)}}\|g\|_{{L_x^{q_2}(\R^2)}}.\end{equation}

$(ii)$ $($Fractional chain rule$)$ Let $G\in
C^1(\mathbb{C}),~s\in(0,1],$ and $1<r,r_1,r_2<+\infty$ satisfy
$\frac1r=\frac1{r_1}+\frac1{r_2}.$ Then
\begin{equation}\label{fraclsfz}
\big\||\nabla|^s
G(u)\big\|_r\lesssim\|G'(u)\|_{r_1}\big\||\nabla|^su\big\|_{r_2}.
\end{equation}
\end{lemma}

As noted in the introduction, one needs to estimate the commutator
 $|Iu|^pIu-I(|u|^pu)$ in the increment of modified
energy $E(Iu)(t)$. When $p$ is an even integer, one can use
multilinear analysis to  expand this commutator into a product of
Fourier transforms of $u$ and $Iu$,  and carefully measure
frequency interactions to derive an estimate (see for example
\cite{CGT09}). However, this is not possible when $p$ in not an even
integer. Instead, Visan and Zhang in \cite{VZ} established the following
rougher  estimate:

\begin{lemma}[\cite{VZ}]\label{commutor}
Let $1<r,r_1,r_2<\infty$ be such that
$\frac1r=\frac1{r_1}+\frac1{r_2}$ and let $0<\nu<s.$ Then,
\begin{equation}\label{equ2.4}
\big\|I(fg)-(If)g\big\|_{L^r}\lesssim
N^{-(1-s+\nu)}\|If\|_{L^{r_1}}\big\|\langle\nabla\rangle^{1-s+\nu}g\big\|_{L^{r_2}}.
\end{equation}
Furthermore, we have
\begin{align}\label{equ2.5} \big\|\nabla
If(u)-(I\nabla u)f'(u)\big\|_{L^r}\lesssim&N^{-(1-s+\nu)}\|\nabla
Iu\|_{L^{r_1}}\big\|\langle\nabla\rangle^{1-s+\nu}f'(u)\big\|_{L^{r_2}},
\end{align}
and
\begin{align}\label{equ2.6} \big\|\nabla If(u)\big\|_{L^r}\lesssim\|\nabla
Iu\|_{L^{r_1}}\|f'(u)\|_{L^{r_2}}+N^{-1+s-\nu}\|\nabla
Iu\|_{L^{r_1}}\|\langle\nabla\rangle^{1-s+\nu}f'(u)\|_{L^{r_2}}.
\end{align}
\end{lemma}

Finally, we conclude this section by recalling the interaction
Morawetz estimate for a solution to problem \eqref{equ1.1}.

\begin{lemma}[Interaction Morawetz estimates
\cite{CGT09}\cite{PV}]\label{lemm2.7}
 Let $u$ be an $H^\frac12$-solution to
\eqref{equ1.1} on the spacetime slab $I\times\R^2$. Then, for any
$t_0\in I$, we have
\begin{equation}\label{equ2.12}
\|u\|_{L_t^4L_x^8(I\times\R^2)}^4\lesssim\|u\|_{L_t^\infty(I;\dot
H^\frac12(\R^2))}^2\|u(t_0)\|_2^2.
\end{equation}
Moreover, interpolating with $\|u\|_{L_t^\infty L_x^2}$, we obtain
\begin{equation}\label{equ2.13}
\|u\|_{L_{t,x}^5(I\times\R^2)}^5\lesssim\|u\|_{L_t^\infty(I;\dot
H^\frac12(\R^2))}^2\|u(t_0)\|_2^3.
\end{equation}

\end{lemma}

\begin{remark}
We adopt the  $L_{t,x}^5$ interaction Morawetz norm, but not the
$L_t^4L_x^8$-norm as used in \cite{CGT09}. As we will see in the
next section, one needs  more restriction on  $p$  to use   the $L_t^4L_x^8$ norm  instead of the
$L_{t,x}^5$ norm. We refer
reader to Remark \ref{rem3.2} for more details.
\end{remark}

To treat the case of low power $p$, we  also need the following improved interaction Morawetz
inequalities.

\begin{lemma}[Improved interaction Morawetz estimates
\cite{CGT07}]\label{lemm2.9}
 Let $u$ be an $H^\frac12$-solution to
\eqref{equ1.1} on the spacetime slab $I\times\R^2$. Then, for any
$t_0\in I$
\begin{equation}\label{equ2.15}
\|u(t,x)\|_{L_{t,x}^4(I\times\R^2)}^4\lesssim
T^\frac13\|u\|_{L_t^\infty(I;\dot
H^\frac12(\R^2))}^2\|u(t_0)\|_2^2+T^\frac13\|u(t_0)\|_2^4.
\end{equation}

\end{lemma}




\section{Proof of Theorem \ref{theorem}}
In this section, we will use the I-method and the interaction Morawetz
estimate to prove Theorem \ref{theorem}. At the first step, we need
to show that the modified energy $E(Iu)$
\begin{equation}
E(Iu)(t)=\tfrac12\int_{\R^2}|\nabla
Iu(t)|^2+\tfrac1{p+2}\int_{\R^2}|Iu(t)|^{p+2}
\end{equation}
 is an ``almost conserved"
quantity in the sense that its derivative  decays with respect to
$N$. In  the following, we always assume $s<1$.

\subsection{Almost Conservation Law}

The aim of this subsection is to control the growth in time of
$E(Iu)(t)$. First, We define $Z_I(t)$ by
 \begin{equation}\label{equ3.1}
 Z_I(t):=\|Iu\|_{Z(t)}=\sup_{(q,r)\in\Lambda_0}\Big(\sum_{N\geq1}\|\nabla
 P_NIu(t)\|_{L_t^qL_x^r([t_0,t)\times\R^2)}^2\Big)^\frac12
 \end{equation}
with convention that $P_1=P_{\leq1}.$  We have the following control of
$Z_I(t)$.

\begin{proposition}[The control of
$Z_I(t)$]\label{zitinc} Let $u(t,x)$ be an $H^s$ solution to
problem \eqref{equ1.1} with $f(u)=|u|^pu$ defined on $[t_0,T]\times\R^2$ and such that
\begin{equation}\label{equ3.2}
\|u\|_{L_{t,x}^{5}([t_0,T]\times\R^2)}\leq\eta
\end{equation}
for some small constant $\eta.$ Assume $E(Iu(t_0))\leq1$. Then for
$s>\tfrac{1+s_c}2$, $p\geq\tfrac{5}{2}$, and sufficiently large $N,$
we have for any $t\in[t_0,T]$
\begin{equation}\label{equ3.3}
Z_I(t)\lesssim \|\nabla Iu(t_0)\|_2+g(t)^pZ_I(t)+
N^{-(s-s_c)}Z_I(t)g(t)^{p-1}\big[g(t)+h(t)\big],
\end{equation}
where $g(t)$ and $h(t)$ are defined as follows:
\begin{equation}\label{equ3.4a}
g(t)^p=\eta^{\theta_1p}\sup_{s\in[t_0,t]}E(Iu(s))^{\frac{(1-\theta_1)p}{p+2}}+\eta^{\theta_2p}Z_I(t)^{(1-\theta_2)p}+
N^{-(1-s_c)p}Z_I(t)^p
\end{equation}
and
\begin{equation}\label{equ3.4.1}
h(t)=\eta^{\theta_1}\sup_{s\in[t_0,t]}E(Iu(s))^{\frac{1-\theta_1}{p+2}}+Z_I(t)
\end{equation}
 with $\theta_1=\tfrac{5}{2p}$,
 $\theta_2=\tfrac{5}{2(3p-5)}\in(0,1)$.
\end{proposition}

\begin{proof}
Applying the operator $I$ to  \eqref{equ1.1} and using
the Strichartz estimate,  we obtain  for all $t\in[t_0,T]$
\begin{align}\nonumber
Z_I(t)\lesssim&\|\nabla Iu(t_0)\|_2+\|\nabla
If(u)\|_{L_{t,x}^\frac{4}{3}([t_0,t)\times\R^2)}\\\nonumber
\lesssim&\|\nabla Iu(t_0)\|_2+\|(\nabla
Iu)f'(u)\|_{L_{t,x}^\frac{4}{3}([t_0,t)\times\R^2)}\\\nonumber
&+\big\|\nabla IF(u)-(\nabla
Iu)f'(u)\big\|_{L_{t,x}^\frac{4}{3}([t_0,t)\times\R^2)}\\\label{equ3.4}
\triangleq&\|\nabla Iu(t_0)\|_2+II_1+II_2.
\end{align}
Throughout the following proof all spacetime norms will be computed on
$[t_0,t)\times\R^2$.

{\bf $\bullet$ The estimate of the term $II_1$:} Using the H\"older and
Minkowski inequalities, we get \begin{equation}\label{equ3.5}
II_1=\|(\nabla Iu)f'(u)\|_{L_{t,x}^\frac{4}{3}}\lesssim\|\nabla
Iu\|_{L_{t,x}^4}\|u\|_{L_{t,x}^{2p}}^p\lesssim
Z_I(t)\|u\|_{L_{t,x}^{2p}}^p.
\end{equation} It remains to estimate $\|u\|_{L_{t,x}^{2p}}^p.$ We decompose
$u=u_{\leq1}+u_{1\leq\cdot\leq N}+u_{>N}$.
We estimate the low frequency part  by an interpolation and the Bernstein inequality
\begin{align}\nonumber
\|u_{\leq1}\|_{L_{t,x}^{2p}}^p\lesssim&\|u_{\leq1}\|_{L_{t,x}^{5}}^{\theta_1p}\|u_{\leq1}\|_{L_{t,x}^\infty}^{(1-\theta_1)p}\lesssim
\eta^{\theta_1p}\|u_{\leq1}\|_{L_t^\infty
L_x^{p+2}}^{(1-\theta_1)p}\\\label{equ3.1d}
\lesssim&\eta^{\theta_1p}\sup_{s\in[t_0,t]}E(Iu(s))^{\frac{(1-\theta_1)p}{p+2}},
\end{align}
where we have used the condition  $p\ge \tfrac{5}{2}$ and $\theta_1=\tfrac{5}{2p}$.

 For the medium frequency
part, we use an interpolation, the Sobolev embedding and the Bernstein
inequality to estimate
\begin{align}\nonumber
\|u_{1\leq\cdot\leq
N}\|_{L_{t,x}^{2p}}^p\lesssim&\|u_{1\leq\cdot\leq
N}\|_{L_{t,x}^{5}}^{\theta_2p}\|u_{1\leq\cdot\leq
N}\|_{L_{t,x}^{3p}}^{(1-\theta_2)p}\lesssim\eta^{\theta_2p}\big\||\nabla|^{1-\frac{4}{3p}}u_{1\leq\cdot\leq
N}\big\|_{L_t^{3p}L_x^\frac{6p}{3p-2}}^{(1-\theta_2)p}\\\label{equ3.1d1}
\lesssim&\eta^{\theta_2p}\big\|\nabla u_{1\leq\cdot\leq
N}\big\|_{L_t^{3p}L_x^\frac{6p}{3p-2}}^{(1-\theta_2)p}\lesssim\eta^{\theta_2p}Z_I(t)^{(1-\theta_2)p},
\end{align}
where $\theta_2=\tfrac{5}{2(3p-5)}$ and
$(3p,\tfrac{6p}{3p-2})\in\Lambda_0$.

 For the high frequency
part, we use Sobolev embedding and \eqref{equ2.2} with $\sigma=s_c$
to obtain
\begin{align}\nonumber
\|u_{> N}\|_{L_{t,x}^{2p}}^p\lesssim&\big\||\nabla|^{s_c}u_{>
N}\big\|_{L_t^{2p}L_x^\frac{2p}{p-1}}^p\lesssim
N^{-(1-s_c)p}\|\nabla
Iu\|_{L_t^{2p}L_x^\frac{2p}{p-1}}^p\\\label{equ3.1d2} \lesssim&
N^{-(1-s_c)p}Z_I(t)^p.
\end{align}

\noindent Thus, collecting \eqref{equ3.5}-\eqref{equ3.1d2} yields
\begin{equation*}
\|u\|_{L_{t,x}^{2p}}^p\lesssim g(t)^p,
\end{equation*}
and so
\begin{equation}\label{gujiagj}
II_1\lesssim Z_I(t)g(t)^p.
\end{equation}

{\bf $\bullet$ The estimate of the term $II_2$:} By the assumption
$s>\tfrac{1+s_c}2$, we have $\nu:=2s-s_c-1<s$. Thus, we deduce from
Lemma \ref{fensc}, \eqref{gujiagj} and \eqref{equ2.5}
\begin{align}\nonumber
II_2=&\big\|\nabla If(u)-(\nabla
Iu)f'(u)\big\|_{L_{t,x}^\frac43}\\\nonumber
\lesssim&N^{-(s-s_c)}\|\nabla
Iu\|_{L_{t,x}^4}\big\|\langle\nabla\rangle^{s-s_c}f'(u)\big\|_{L_{t,x}^2}\\\nonumber
\lesssim&N^{-(s-s_c)}Z_I(t)\Big(\|f'(u)\|_{L_{t,x}^{2}}+
\big\||\nabla|^{s-s_c}f'(u)\big\|_{L_{t,x}^{2}}\Big)\\\nonumber
\lesssim&N^{-(s-s_c)}Z_I(t)\Big(\|u\|_{L_{t,x}^{2p}}^p+\|u\|_{L_{t,x}^{2p}}^{p-1}\big\||\nabla|^{s-s_c}
u\big\|_{L_{t,x}^{2p}}\Big)\\\nonumber
\lesssim&N^{-(s-s_c)}Z_I(t)g(t)^{p-1}\Big[g(t)+\big\||\nabla|^{s-s_c}
u\big\|_{L_{t,x}^{2p}}\Big]\\\label{equ3.6.1d1}
\lesssim&N^{-(s-s_c)}Z_I(t)g(t)^{p-1}\big[g(t)+h(t)\big],
\end{align}
where  we have used the same argument as deriving
\eqref{gujiagj} to estimate
\begin{align}\nonumber\big\||\nabla|^{s-s_c}
u\big\|_{L_{t,x}^{2p}}\lesssim&\big\|
u_{\leq1}\big\|_{L_{t,x}^{2p}}+\big\||\nabla|^{s-s_c}
u_{1\leq\cdot\leq N}\big\|_{L_{t,x}^{2p}}+\big\||\nabla|^{s-s_c}
u_{>N}\big\|_{L_{t,x}^{2p}}\\\nonumber
\lesssim&~\|u_{\leq1}\|_{L_{t,x}^{2p}}+\|\nabla u_{1\leq\cdot\leq
N}\|_{L_t^{2p}L_x^{\frac{2p}{p-1}}}+\big\||\nabla|^{s}u_{>N}\big\|_{L_t^{2p}L_x^{\frac{2p}{p-1}}}\\\nonumber
\lesssim&~\eta^{\theta_1}\sup_{s\in[t_0,t]}E(Iu(s))^{\frac{1-\theta_1}{p+2}}+Z_I(t)+N^{s-1}Z_I(t)\\\label{equbuch}
\lesssim&~h(t).
\end{align}
 This estimate together with inequality
\eqref{gujiagj} ends the proof of Proposition \ref{zitinc}.
\end{proof}

\begin{remark}\label{rem3.2}
We  assume the $L_{t,x}^5$ interaction Morawetz norm to be
small unlike  the small  interaction Morawetz norm $L_t^4L_x^8$  as
used in \cite{CGT09}.  If we replace
\eqref{equ3.2} by
\begin{align}\label{equ3.12d536}
\|u\|_{L_t^4L_x^8([t_0,T]\times\R^2)}\leq\eta,\end{align}
then, by the similar argument as above, we need the estimate
\begin{equation*}
\big\|(\nabla Iu)f'(u)\big\|_{L_t^{q_1'}L_x^{r_1'}}\lesssim\|\nabla
Iu\|_{L_t^qL_x^r}\|u\|_{L_t^\frac{p}{\theta}L_x^\frac{p}{1-\theta}}^p,\quad
(q_1,r_1),~(q,r)\in\Lambda_0,~\theta\in[0,1]
\end{equation*}
together with the low frequency part
$$
\|u_{\leq1}\|_{L_t^\frac{p}{\theta}L_x^\frac{p}{1-\theta}}^p\lesssim\|u_{\leq1}\|_{L_t^4L_x^8}^{4\theta}\|u_{\leq1}\|_{L_{t,x}^\infty
}^{p-4\theta},
$$
where we need the restriction $\frac{p}\theta\geq4$ and
$\frac{p}{1-\theta}\geq8$. Therefore,
$$p\geq\min_{0\leq\theta\leq1}\max\{4\theta,8(1-\theta)\}=\tfrac83.$$
This argument, compared with $p\geq\tfrac52$ in  Proposition \ref{zitinc},
 shows that $L_{t,x}^5$ is better than  $L_t^4L_x^8$.
\end{remark}

Next, we show the energy increment of $E(Iu)(t)$.

\begin{proposition}[Energy increment]\label{enerin}
Let $u(t,x)$ be an $H^s$ solution to \eqref{equ1.1} with
$f(u)=|u|^pu$ defined on $[t_0,T]\times\R^2$, which  satisfies
\begin{equation}\label{equ3.2a}
\|u\|_{L_{t,x}^{5}([t_0,T]\times\R^2)}\leq\eta
\end{equation}
for some small constant $\eta.$ Assume $E(Iu(t_0))\leq1$. Then for
$s\geq\tfrac{p}{p+1}$, $p\geq\tfrac{11}{4}$ and sufficiently large
$N,$  we have for any $t\in[t_0,T]$
\begin{align}\label{equ3.7}
&\big|\sup_{s\in[t_0,t]}E(Iu(s))-E(Iu(t_0))\big|\\\nonumber
\lesssim&N^{-(1-s_c)}Z_I(t)^3g(t)^{p-1}+N^{-(s-s_c)}Z_I(t)^2g(t)^{p-1}\big[g(t)+h(t)\big]\\\nonumber
&+ k(t)\Big\{N^{-1}g(t)^pZ_I(t)+\eta
N^{-(s-\frac12+\frac1q)}m(t)\Big\},
\end{align}
where $k(t)$ and $m(t)$  are defined by
\begin{align*}
k(t)=&\eta^{\theta_3(p+1)}\sup_{s\in[t_0,t]}E(Iu(s))^{\frac{(1-\theta_3)(p+1)}{p+2}}+
Z_I(t)^{p+1},~\theta_3=\tfrac{5}{3(p+1)}\end{align*} and
\begin{align*}
m(t)=&\eta^{1+\theta_4}g(t)^{(1-\theta_4)(p-1)}\Big(\eta^{\theta_4}
g(t)^{1-\theta_4}+Z_I(t)\Big)
\end{align*}
with $\theta_4=\tfrac{1}{4p-10}$ and $q=\tfrac{5p}{11}$.

\end{proposition}

\begin{proof}
Since $iIu_t+\Delta Iu=IF(u),$ we get by a simple computation
\begin{align*}
E(Iu(t))-E(Iu(t_0))=&\int_{t_0}^t\frac{\pa}{\pa s}E(Iu(s))ds\\
=&{\rm Re}\int_{t_0}^t\int_{\R^d}\overline{Iu_t}\big(-\Delta
Iu+f(Iu)\big)dxds\\
=&{\rm
Re}\int_{t_0}^t\int_{\R^d}\overline{Iu_t}\big[f(Iu)-If(u)\big]dxds\\
=&-{\rm Im}\int_{t_0}^t\int_{\R^d}\overline{\nabla Iu}\cdot\nabla\big[f(Iu)-If(u)\big]dxds\\
&-{\rm
Im}\int_{t_0}^t\int_{\R^d}\overline{If(u)}\big[f(Iu)-If(u)\big]dxds\\
=&-{\rm Im}\int_{t_0}^t\int_{\R^d}\overline{\nabla Iu}\cdot\nabla Iu\big[f'(Iu)-f'(u)\big]dxds\\
&-{\rm Im}\int_{t_0}^t\int_{\R^d}\overline{\nabla
Iu}\cdot\Big[(\nabla
Iu)f'(u)-I(f'(u)\nabla u)\Big]dxds\\
&-{\rm Im}\int_{t_0}^t\int_{\R^d}\overline{If(u)}\big[f(Iu)-If(u)\big]dxds\\
\triangleq&II_1+II_2+II_3.
\end{align*}

{\bf $\bullet$ The estimate of $II_1$:}  Since  $(2p,\tfrac{2p}{p-1})\in\Lambda_0$, by inequality  \eqref{equ2.2} with $\sigma=s_c$,
and  the Sobolev embedding, we estimate
\begin{align}\nonumber
|II_1|\lesssim&\|\nabla
Iu\|_{L_{t,x}^4}^2\|f'(Iu)-f'(u)\|_{L_{t,x}^2}\\\nonumber \lesssim&
Z_I(t)^2\|P_{>N}u\|_{L_{t,x}^{2p}}\|u\|_{L_{t,x}^{2p}}^{p-1}\\\nonumber
\lesssim&Z_I(t)^2\big\||\nabla|^{s_c}P_{>N}u\big\|_{L_t^{2p}L_x^{\frac{2p}{p-1}}}g(t)^{p-1}\\\nonumber
\lesssim&Z_I(t)^2 N^{-(1-s_c)}\|\nabla
Iu\|_{L_t^{2p}L_x^{\frac{2p}{p-1}}}g(t)^{p-1}\\\label{equ3.8}
\lesssim&N^{-(1-s_c)}Z_I(t)^3g(t)^{p-1}.
\end{align}

{\bf $\bullet$ The estimate of $II_2$:} Using H\"older's inequality
and \eqref{equ3.6.1d1}, we obtain
\begin{align}\nonumber
|II_2|\lesssim&\|\nabla Iu\|_{L_{t,x}^4}\big\|(\nabla
Iu)f'(u)-I(f'(u)\nabla
u)\big\|_{L_{t,x}^\frac{4}{3}}\\\label{equ3.9}
\lesssim&N^{-(s-s_c)}Z_I(t)^2g(t)^{p-1}\big[g(t)+h(t)\big].
\end{align}

{\bf $\bullet$ The estimate of $II_3$:} By H\"older's inequality and
Minkowski's inequality, we get
\begin{align}\nonumber
|II_3|\lesssim&\|If(u)\|_{L_{t,x}^4}\|f(Iu)-If(u)\|_{L_{t,x}^\frac{4}{3}}\\\label{equ3.10}
\lesssim&\|u\|_{L_{t,x}^{4
(p+1)}}^{p+1}\Big\{\big\|Iu(|Iu|^p-|u|^p)\big\|_{L_{t,x}^\frac{4}{3}}+\big\|(Iu)|u|^p-I(|u|^pu)\big\|_{L_{t,x}^\frac{4}{3}}\Big\}.
\end{align}
To estimate $\|u\|_{L_{t,x}^{4(p+1)}}^{p+1}$, we decompose
$u=u_{\leq1}+u_{1\leq\cdot\leq N}+u_{>N}$. Using the same argument
as leading to \eqref{gujiagj} and
$(4(p+1),\tfrac{4(p+1)}{2p+1})\in\Lambda_0$, one has  the inequality by
means of \eqref{equ2.2} with $\sigma=\frac{p}{p+1}$,
\begin{align}\nonumber
&\|u\|_{L_{t,x}^{4(p+1)}}^{p+1}\\\nonumber
\lesssim&\|u_{\leq1}\|_{L_{t,x}^{4(p+1)}}^{p+1}+\|u_{1\leq\cdot\leq
N}\|_{L_{t,x}^{4(p+1)}}^{p+1}+\|u_{>N}\|_{L_{t,x}^{4(p+1)}}^{p+1}\\\nonumber
\lesssim&\|u\|_{L_{t,x}^{5}}^{\theta_3(p+1)}\|u_{\leq1}\|_{L_{t,x}^\infty}^{(1-\theta_3)(p+1)}+\big\||\nabla|^\frac{p}{p+1)}u_{1\leq\cdot\leq
N}\big\|_{L_t^{4(p+1)}L_x^\frac{4(p+1)}{2p+1}}^{p+1}+\big\||\nabla|^\frac{p}{p+1}u_{\geq
N}\big\|_{L_t^{4(p+1)}L_x^\frac{4(p+1)}{2p+1}}^{p+1} \\\nonumber
\lesssim&\eta^{\theta_3(p+1)}\|u_{\leq1}\|_{L_t^\infty
L_x^{p+2}}^{(1-\theta_3)(p+1)}+\|\nabla u_{1\leq\cdot\leq
N}\|_{L_t^{4(p+1)}L_x^\frac{4(p+1)}{2p+1}}^{p+1}+N^{(\frac{p}{p+1}-1)(p+1)}\|\nabla
Iu\|_{L_t^{4(p+1)}L_x^\frac{4(p+1)}{2p+1}}^{p+1}
\\\nonumber
\lesssim&\eta^{\theta_3(p+1)}\sup_{s\in[t_0,t]}E(Iu(s))^{\frac{(1-\theta_3)(p+1)}{p+2}}+
Z_I(t)^{p+1}\\\label{equ3.11} \simeq& k(t),
\end{align}
where $\theta_3=\tfrac{5}{3(p+1)}$ and where we have used the condition $s\geq\frac{p}{p+1}$.

To estimate $\big\|Iu(|Iu|^p-|u|^p)\big\|_{L_{t,x}^\frac{4}{3}}$, we
use H\"older's inequality, \eqref{equ2.2} with $\sigma=0$ and
\eqref{gujiagj} to get
\begin{align}\nonumber
\big\|Iu(|Iu|^p-|u|^p)\big\|_{L_{t,x}^\frac{4}{3}}\lesssim&\|Iu\|_{L_{t,x}^{2p}}\|u\|_{L_{t,x}^{2p}}^{p-1}\|u_{\geq
N}\|_{L_{t,x}^4}\\\nonumber
\lesssim&\|u\|_{L_{t,x}^{2p}}^pN^{-1}\|\nabla
Iu\|_{L_{t,x}^4}\\\label{equ3.12} \lesssim&N^{-1}g(t)^pZ_I(t).
\end{align}

To estimate $\big\|(Iu)|u|^p-I(|u|^pu)\big\|_{L_{t,x}^\frac{4}{3}}$,
using \eqref{equ2.4} with $\nu=2s-\tfrac32+\tfrac1q\in(0,s)$ and
\eqref{fraclsfz}, we obtain for $q=\tfrac{5p}{11}$
\begin{align}\label{equ3.13}
&\big\|(Iu)|u|^p-I(|u|^pu)\big\|_{L_{t,x}^\frac{4}{3}}\\\nonumber
\lesssim&N^{-(s-\frac12+\frac1q)}\|Iu\|_{L_{t,x}^5}\big\|\langle\nabla\rangle^{s-\frac12+
\frac1q}|u|^p\big\|_{L_{t,x}^{\frac{4q}{p}}}\\\nonumber
\lesssim&N^{-(s-\frac12+\frac1q)}\eta\|u\|_{L_{t,x}^{4q}}^{p-1}
\Big(\|u\|_{L_{t,x}^{4q}}+\big\||\nabla|^{
s-\frac12+\frac1q}u\big\|_{L_{t,x}^{4q}}\Big)\\\nonumber
\lesssim&\eta^{1+\theta_4(p-1)}N^{-(s-\frac12+\frac1q)}g(t)^{(1-\theta_4)(p-1)}\Big(\eta^{\theta_4}
g(t)^{1-\theta_4}+Z_I(t)\Big)\\\nonumber
\lesssim&N^{-(s-\frac12+\frac1q)}m(t),
\end{align}
where   we have used the estimates
\begin{align}\label{restricti}
\|u\|_{L_{t,x}^{4q}}\lesssim\|u\|_{L_{t,x}^{5}}^{\theta_4}\|u\|_{L_{t,x}^{2p}}^{1-\theta_4}\lesssim\eta^{\theta_4}
g(t)^{1-\theta_4},~\theta_4=\tfrac{1}{4p-10}
\end{align}
and
\begin{align*}
\big\||\nabla|^{
s-\frac12+\frac1q}u\big\|_{L_{t,x}^{4q}}\lesssim&\|u_{\leq1}\|_{L_{t,x}^{4q}}+\big\||\nabla|^su_{1\leq\cdot\leq
N}\big\|_{L_t^{4q}L_x^r}+\big\||\nabla|^su_{\geq
N}\big\|_{L_t^{4q}L_x^r}\\
\lesssim&\eta^{\theta_4} g(t)^{1-\theta_4}+Z_I(t)+N^{1-s}Z_I(t),
\end{align*}
with $r=\tfrac{20p}{5p-11},~(4q,r)\in\Lambda_0.$  Since inequality  $p\geq\tfrac{11}{4}$  guarantees $4q\geq 5$,
the interpolation inequality in \eqref{restricti} is valid.
Thus, plugging \eqref{equ3.11}-\eqref{equ3.13} into \eqref{equ3.10},
we get
\begin{equation}\label{equ3.14}
|II_3|\lesssim k(t)\Big\{N^{-1}g(t)^pZ_I(t)+
N^{-(s-\frac12+\frac1q)}m(t)\Big\}.
\end{equation}
 This estimate,  together with  \eqref{equ3.8}, \eqref{equ3.9},  yields \eqref{equ3.7}.
\end{proof}

\vskip 0.2cm

Now we use  a  standard bootstrap argument to  show that the quantity $E(Iu)(t)$ is
``almost conserved"  by making use of inequality
$$s-s_c\leq\min\{1-s_c,~s-\tfrac12+\tfrac1q\}.$$

\begin{proposition}[Almost conservation law]\label{almost}
Let $u(t,x)$ be an $H^s$ solution to problem \eqref{equ1.1} with
$f(u)=|u|^pu$ defined on $[t_0,T]\times\R^2$ and  satisfy
\begin{equation}\label{equ3.2.1}
\|u\|_{L_{t,x}^{5}([t_0,T]\times\R^2)}\leq\eta
\end{equation}
for some small constant $\eta.$ Assume $E(Iu(t_0))\leq1$. Then for
$$s\geq\max\big\{\tfrac{1+s_c}2,~\tfrac{p}{p+1}\big\},\quad
p\geq\tfrac{11}{4}$$
and sufficiently large $N,$ we have
\begin{align}\label{equ3.3.1}
E(Iu)(t)=E(Iu(t_0))+O(N^{s_c-s}).
\end{align}

\end{proposition}

\begin{proof} Expression \eqref{equ3.3.1} will follow from Proposition \ref{zitinc} and
\ref{enerin} provided   we  establish
\begin{equation}\label{equ3.15}
Z_I(t)\lesssim1\quad \text{and}\quad
\sup_{s\in[t_0,t]}E(Iu(s))\lesssim1,\quad\forall~ t\in[t_0,T].
\end{equation}
From the assumption $E(Iu(t_0))\leq1$, we only need to prove that
\begin{equation}\label{equ3.16}
Z_I(t)\lesssim\|\nabla Iu(t_0)\|_2,\quad\forall~t\in[t_0,T]
\end{equation}
and
\begin{equation}\label{equ3.17}
\sup_{s\in[t_0,t]}E(Iu(s))\lesssim
E(Iu(t_0)),\quad\forall~t\in[t_0,T].
\end{equation}
We show it  by a  standard  bootstrap argument. It suffices  to show that the above two properties hold  on the interval $[t_0,T]$.
Let
\begin{align*}
\Omega_1:=&\big\{t\in[t_0,T]:~Z_I(t)\leq C_1\|\nabla
Iu_0\|_2,~\sup_{s\in[t_0,t]}E(Iu(s))\leq C_2E(Iu(t_0))\big\},\\
\Omega_2:=&\big\{t\in[t_0,T]:~Z_I(t)\leq 2C_1\|\nabla
Iu_0\|_2,~\sup_{s\in[t_0,t]}E(Iu(s))\leq 2C_2E(Iu(t_0))\big\},
\end{align*}
where  $C_1$ and $C_2$  are sufficiently large
 constants which may depend  on  the Strichartz constant.

In order to run the bootstrap argument successfully, we need to
verify three properties:
\begin{enumerate}
\item $\Omega_1$ is a nonempty closed set.

\item $\Omega_2\subset\Omega_1$.

\item If $t\in\Omega_1$, then there exists $\varepsilon>0$ such that
$[t,t+\varepsilon)\subset\Omega_2$.
\end{enumerate}

In fact,  since $t_0\in\Omega_1$, one easily verifies $\Omega_1$ is a nonempty closed by Fatou's Lemma.
Combining  Proposition \ref{zitinc} and
\ref{enerin}  yields (2) by taking $N$ sufficiently large and $\eta$ sufficiently small
depending on $C_1,~C_2$ and $E(Iu(t_0))$.
property (3)   follows from (2) and  from the local
well-posedness theory.

\vskip0.12cm

The last two statements show that $\Omega_1$ is open from the right-hand side and
Proposition \ref{almost} is proved.
\end{proof}

\subsection{Global well-posedness}
In this part, we  establish the global time-space estimates  in terms of a rough norm of initial data
by making use of  the interaction Morawetz estimate  and  almost conservation
law with a scaling argument.

\begin{proposition}\label{promain}
Suppose $u(t,x)$ is a global solution to problem \eqref{equ1.1} with
$f(u)=|u|^pu$ satisfying  $u_0\in C_0^\infty(\R^2)$. Then for
$$\max\big\{\tfrac{p}{p+1},\tfrac{1+s_c}2,s_1\big\}<s<1~~\text{and}~~
p\geq\tfrac{11}{4},$$ we have
\begin{align}\label{equ4.1}
\|u\|_{L_{t,x}^{5}(\R\times\R^2)}\leq&
C\big(\|u_0\|_{H^s(\R^2)}\big),\\\label{equ4.2}
\sup_{t\in\R}\|u(t)\|_{H^s(\R^2)}\leq&C\big(\|u_0\|_{H^s(\R^2)}\big),
\end{align}
where $s_1$ is the positive root of the quadratic equation
$$s^2+2s_cs+s_c^2-4s_c=0.$$
\end{proposition}

\begin{remark} From the local well-posedness theory, we know that the lifespan of a
 local solution depends only on the
$H^s$-norm of the initial data. Thus, the global well-posedness part
of Theorem \ref{theorem} follows from \eqref{equ4.2} and the standard density argument.
\end{remark}

\begin{proof}[ Proof of Proposition  \eqref{promain}]
If $u$ is a solution to problem \eqref{equ1.1} with $f(u)=|u|^pu$, so
is
\begin{equation}\label{equ4.3}
u^\la(t,x)=\la^{-\frac2p}u\big(\tfrac{t}{\la^2},\tfrac{x}{\la}\big).
\end{equation}
By inequality \eqref{equ2.3} and the Sobolev embedding, we have
\begin{align*}
\|\nabla
 Iu_0^\la\|_{L^2(\R^2)}\lesssim&N^{1-s}\|u_0^\la\|_{\dot H^s}\simeq
 N^{1-s}\la^{s_c-s}\|u_0\|_{\dot H^s},\\
\|Iu_0^\la\|_{L^{p+2}}\lesssim&\|u_0^\la\|_{L^{p+2}}=\la^{-\frac2p+\frac
2{p+2}}\|u_0\|_{p+2}\lesssim\la^{-\frac2p+\frac{2}{p+2}}\|u_0\|_{H^s}.
\end{align*}
As $s>s_c$, taking $\la$ sufficiently large depending on
$\|u_0\|_{H^s}$ and $N$ such that
\begin{equation}\label{equ4.4}
N^{1-s}\la^{s_c-s}\|u_0\|_{\dot H^s}\ll1\quad\text{and}\quad
\la^{-\frac2p+\frac{2}{p+2}}\|u_0\|_{H^s}\ll1,
\end{equation}
we get
\begin{equation}\label{assimlar}
E(Iu_0^\la)\ll1.
\end{equation}

Next, we claim that there exists an absolute constant $C$ such that
\begin{equation}\label{equ4.5}
\|u^\la\|_{L_{t,x}^{5}(\R\times\R^2)}\leq C\la^{\frac{4}{5}s_c}.
\end{equation}
Choosing $\lambda=1$ yields \eqref{equ4.1}. We prove  inequality \eqref{equ4.5}
via a bootstrap argument. By time reversal symmetry, it suffices to
argue for positive time only. Define
$$\Omega_1:=\big\{t\in[0,\infty):~\|u^\la\|_{L_{t,x}^{5}([0,t]\times\R^2)}\leq C\la^{\frac{4}{5}s_c}\big\}.$$
Our goal is to prove $\Omega_1=[0,\infty)$. Let
$$\Omega_2:=\big\{t\in[0,\infty):~\|u^\la\|_{L_{t,x}^{5}([0,t]\times\R^2)}\leq 2C\la^{\frac{4}{5}s_c}\big\}.$$
In order to run the bootstrap argument successfully, we need to
check the following properities:
\begin{enumerate}
\item $\Omega_1$ is  a nonempty  closed(as $0\in\Omega_1$ and using Fatou's Lemma);
\item $\Omega_2\subset\Omega_1$;
\item If $t\in\Omega_1$, then there exists $\varepsilon>0$ such that
$[t,t+\varepsilon)\subset\Omega_2$.
\end{enumerate}
Property (3) follows from (2) and the local
well-posedness theory.  Thus, it suffices  to prove (2): For any
$T\in\Omega_2$, we want to show that $T\in\Omega_1$. Throughout the
following proof, all the space-time norms will be computed  on $[0,T]\times\R^2$.

Using the interaction Morawetz estimate and the mass conservation,
we get
\begin{equation}\label{equ4.6d1}
\|u^\la\|_{L_{t,x}^5}\lesssim\|u_0^\la\|_2^\frac35\|u^\la\|_{L_t^\infty
\dot
H^\frac12}^\frac{2}{5}\lesssim\la^{\frac{3}{5}s_c}\|u_0\|_2^\frac3{5}\|u^\la\|_{L_t^\infty
\dot H^\frac12}^\frac{2}{5}.
\end{equation}
To control the term  $\|u^\la(t)\|_{L_t^\infty\dot H^\frac12}$, we
decompose $u^\la=P_{\leq N}u^\la+P_{>N}u^\la.$

For the low frequency part, we interpolate between the $L_x^2$-norm
and $\dot H^1$-norm and we use the fact that the operator $I$ is the
identity on frequencies $|\xi|\leq N$:
\begin{equation}\label{equ4.7}
\big\|P_{\leq N}u^\la\big\|_{\dot H^\frac12}\lesssim\big\|P_{\leq
N}u^\la\big\|_{2}^\frac12\big\|P_{\leq N}u^\la\big\|_{\dot
H^1}^\frac12\lesssim\la^{\frac12s_c}\|u_0\|_2^\frac12\|Iu^\la\|_{\dot
H^1}^\frac12.
\end{equation}
To estimate the high frequency part, we interpolate between the
$L_x^2$-norm and $\dot H^s$-norm and use \eqref{equ2.2} and
\eqref{equ4.4} to obtain
\begin{align}\nonumber
\big\|P_{>N}u^\la\big\|_{\dot
H^\frac12}\lesssim&\big\|P_{>N}u^\la\big\|_2^{1-\frac1{2s}}\big\|P_{>N}u^\la\big\|_{\dot
H^s}^\frac1{2s}\\\nonumber
\lesssim&\la^{(1-\frac1{2s})s_c}\|u_0\|_2^{1-\frac1{2s}}N^{\frac{s-1}{2s}}\|Iu^\la\|_{\dot
H^1}^\frac1{2s}\\\label{equ4.8}
\lesssim&\la^{s_c-\frac12}\|u_0\|_2^{1-\frac1{2s}}\|Iu^\la\|_{\dot
H^1}^\frac1{2s}.
\end{align}

Plugging \eqref{equ4.7} and \eqref{equ4.8} into \eqref{equ4.6d1}, we
estimate
\begin{align*}
\|u^\la\|_{L_{t,x}^{5}}\lesssim&\la^{\frac{3}{5}s_c}\|u_0\|_2^\frac3{5}\|u^\la\|_{L_t^\infty
\dot H^\frac12}^\frac{2}{5}\\
\lesssim&\la^{\frac{3}{5}s_c}\|u_0\|_2^\frac3{5}\sup_{s\in[0,T]}\Big(\la^{\frac12s_c}\|u_0\|_2^\frac12\|Iu^\la\|_{\dot
H^1}^\frac12+\la^{s_c-\frac12}\|u_0\|_2^{1-\frac1{2s}}\|Iu^\la\|_{\dot
H^1}^\frac1{2s}\Big)^\frac{2}{5}\\
\lesssim&C(\|u_0\|_2)\la^{\frac{4}{5}s_c}\sup_{s\in[0,T]}\Big(\|Iu^\la\|_{\dot
H^1}^\frac12+\|Iu^\la\|_{\dot H^1}^\frac1{2s}\Big)^\frac{2}{5}
\end{align*}
where we have used the fact  $\la\gg1$ in the last inequality.
Thus, choosing  $C$ sufficiently large depending on $\|u_0\|_2$, we
obtain $T\in\Omega_1$ provided we can prove
\begin{equation}\label{equ4.10}
\sup_{s\in[0,T]}\|Iu^\la\|_{\dot H^1}\leq1, \quad T\in\Omega_2.
\end{equation}
 In fact, let  $\eta>0$ be
sufficiently small constant as in Proposition \ref{almost}, and we
divide $[0,T]$ into
\begin{equation}\label{equ4.11}
L\sim\Big(\frac{\la^{\frac{4}{5}s_c}}\eta\Big)^{5}\sim\la^{4s_c},
\end{equation}
subintervals $I_j=[t_j,t_{j+1}]$ such that
\begin{equation}\label{equ4.12}
\|u^\la\|_{L_{t,x}^{5}(I_j\times\R^2)}\leq\eta.
\end{equation}
Using Proposition \ref{almost} on each interval $I_j,$ we obtain
\begin{equation}\label{equ4.13}
\sup_{t\in[0,T]}E(Iu^\la)(t)\leq E(Iu_0^\la)+LN^{s_c-s}.
\end{equation}
To control the changes of energy during the iteration, we need
$$LN^{s_c-s}\simeq\la^{4s_c}N^{s_c-s}\ll1
.$$
 This  fact together with \eqref{equ4.4} leads to
$$N^{4s_c\frac{s-1}{s_c-s}}N^{s_c-s}\simeq
N^{-\frac{s^2+2s_cs+s_c^2-4s_c}{s-s_c}}\ll1.$$ This may be ensured
by taking $N=N(\|u_0\|_{H^s})$ large enough provided that $s$
satisfies
$$s^2+2s_cs+s_c^2-4s_c>0.$$ This can be verified by $s>s_1$  by the definition of  $s_1$. This completes the bootstrap argument, hence
we prove the claim \eqref{equ4.5},  and moreover,  \eqref{equ4.1} follows.

To deal with $\|u(t)\|_{H^s}$, by the conservation of mass,  and inequalities \eqref{equ2.3} and \eqref{equ4.10}, we estimate
\begin{align*}
\|u(t)\|_{H^s}\lesssim&\|u_0\|_{L^2}+\|u(t)\|_{\dot
H^s}\\
\lesssim&\|u_0\|_{L^2}+\la^{s-s_c}\|u^\la(\la^2t)\|_{\dot
H^s}\\
\lesssim&\|u_0\|_{L^2}+\la^{s-s_c}\|Iu^\la(\la^2t)\|_{H^1}\\
\lesssim&\|u_0\|_{L^2}+\la^{s-s_c}\big(\|u_0^\la\|_2+\|Iu^\la(\la^2t)\|_{\dot
H^1}\big)\\
\lesssim&\|u_0\|_{L^2}+\la^{s-s_c}\big(\la^{s_c}\|u_0\|_2+1\big)\\
\leq&C(\|u_0\|_{H^s}).
\end{align*}
This completes the proof of \eqref{equ4.2}.

\end{proof}

\subsection{Scattering}
We prove that the scattering part of Theorem \ref{theorem} holds for
$H^s_x(\R^2)$ with $s\in(s_0,1)$. We first show that the global
Morawetz estimate can be improved to the global Strichartz estimate
\begin{equation}\label{equ4.14}
\|\langle\nabla\rangle^su\|_{S^0(\R)}:=\sup_{(q,r)\in\Lambda_0}\|\langle\nabla\rangle^su\|_{L_t^qL_x^r(\R\times\R^2)}.
\end{equation}
Second, we use this estimate to show  the asymptotic completeness property. Since the
construction of the wave operator is standard,  we omit it here.

Let $u$ be a global solution to problem \eqref{equ1.1}. From the
interaction Morawetz estimate \eqref{equ4.1}, we have
\begin{equation}\label{equ4.15}
\|u\|_{L_{t,x}^{5}(\R\times\R^2)}\leq C(\|u_0\|_{H^s}).
\end{equation}
Let $\eta>0$ be a small constant to be chosen later and split $\R$
into $L=L(\|u_0\|_{H^s})$ subintervals $I_j=[t_j,t_{j+1}]$ such that
\begin{equation}\label{equ4.16}
\|u\|_{L_{t,x}^{5}(I_j\times\R^2)}\leq\eta.
\end{equation}
Using  \eqref{fraclsfz} and  \eqref{equ4.2}, One  gets
\begin{align}\nonumber
\big\|\langle\nabla\rangle^su\big\|_{S^0(I_j)}\lesssim&\|u(t_j)\|_{H^s}+\big\|\langle\nabla\rangle^s(|u|^pu)
\big\|_{L_{t.x}^\frac{4}{3}(I_j\times\R^2)}\\\label{equ4.17}
\lesssim&C(\|u_0\|_{H^s})+\|u\|_{L_{t,x}^{2p}(I_j\times\R^2)}^p\big\|\langle\nabla\rangle^su\big\|_{L_{t,x}^4(I_j\times\R^2)}.
\end{align}
We use the interpolation and the Sobolev embedding to estimate
\begin{align*}
\|u\|_{L_{t,x}^{2p}(I_j\times\R^2)}^p\lesssim&\|u\|_{L_{t,x}^5(I_j\times\R^2)}^{\theta_2p}\|u\|_{L_{t,x}^{3p}(I_j\times\R^2)}
^{(1-\theta_2)p}\\
\lesssim&\eta^{\theta_2p}\big\||\nabla|^{1-\frac{4}{3p}}u\big\|_{L_t^{3p}L_x^\frac{6p}{3p-2}(I_j\times\R^2)}^{(1-\theta_2)p}\\
\lesssim&\eta^{\theta_2p}\big\|\langle\nabla\rangle^su\big\|_{S^0(I_j)}^{(1-\theta_2)p},\quad  \theta_2=\tfrac{5}{2(3p-5)},
\end{align*}
where $(3p,\tfrac{6p}{3p-2})\in\Lambda_0$, and we have used the fact
$s>\frac{1+s_c}2>1-\frac{4}{3p}.$ Hence,
\begin{equation}\label{equ4.18}
\big\|\langle\nabla\rangle^su\big\|_{S^0(I_j)}\lesssim
C(\|u_0\|_{H^s})+\eta^\frac{5}{2}\big\|\langle\nabla\rangle^su\big\|_{S^0(I_j)}^{p-\frac{5}{2}}.
\end{equation}
By a standard continuity argument, we have
\begin{equation}\label{equ4.19}
\|\langle\nabla\rangle^su\|_{S^0(I_j)}\leq C(\|u_0\|_{H^s}),
\end{equation}
provided that we take $\eta$ sufficiently small depending on the
initial data $\|u_0\|_{H^s}$. Summing over all subintervals $I_j$,
we obtain
\begin{equation}\label{equ4.20}
\|\langle\nabla\rangle^su\|_{S^0(\R)}\leq C(\|u_0\|_{H^s}).
\end{equation}

Finally, we utilize this estimate to show the asymptotic completeness. It suffices
 to prove that there exists a unique $u_\pm$ such that
$$\lim_{t\to\pm\infty}\|u(t)-e^{it\Delta}u_\pm\|_{H^s_x}=0.$$
By time reversal symmetry, we only need  to prove it  for positive
times. For $t>0$, we will show that $v(t):=e^{-it\Delta}u(t)$
converges in $H^s_x$ as $t\to+\infty$, and denote $u_+$ to be the
limit. In fact, we obtain by Duhamel's formula
\begin{equation}\label{equ4.21}
v(t)=u_0-i\int_0^te^{-i\tau\Delta}f(u)(\tau)d\tau.
\end{equation}
Hence, for $0<t_1<t_2$, we have
$$v(t_2)-v(t_1)=-i\int_{t_1}^{t_2}e^{-i\tau\Delta}f(u)(\tau)d\tau.$$
By the Strichartz estimate and \eqref{equ4.18}, we deduce that
\begin{align*}
\|v(t_2)-v(t_1)\|_{H^s(\R^2)}=&\Big\|\int_{t_1}^{t_2}e^{-i\tau\Delta}f(u)(\tau)d\tau\Big\|_{H^s(\R^2)}\\
\lesssim&\|u\|_{L_{t,x}^{5}([t_1,t_2]\times\R^2)}^{\theta_2p}\big\|\langle\nabla\rangle^su\big\|_{S^0([t_1,t_2])}^{1+(1-\theta_2)p}
\\
\to&0\quad \text{as}\quad t_1,~t_2\to\infty.
\end{align*}
Thus, the limit of  \eqref{equ4.21} as $+\infty$  is well
defined. In particular, we find that
$$u_+=u_0-i\int_0^\infty e^{-i\tau\Delta}f(u)(\tau)d\tau$$
is nothing but the asymptotic state. Therefore, we completed  the
proof of Theorem \ref{theorem}.




\section{Proof of Theorem \ref{theorem1.2}}
In this section, we will use the I-method and the improved
interaction Morawetz estimates in Lemma \ref{lemm2.9} to show Theorem
\ref{theorem1.2}. The process is similar to the one in the preceding  section.
First, we define
 \begin{equation}\label{equ3.1a}
 Z_I(t):=\|Iu\|_{Z(t)}=\sup_{(q,r)\in\Lambda_0}\Big(\sum_{N\geq1}\|\nabla
 P_NIu(t)\|_{L_t^qL_x^r([t_0,t)\times\R^2)}^2\Big)^\frac12
 \end{equation}
with the convention that $P_1=P_{\leq1}$.
  Then, we have to control $Z_I(t)$ as follows.

\begin{proposition}[The control of
$Z_I(t)$]\label{zitinc6.1} Let $u(t,x)$ be an $H^s$ solution to
problem \eqref{equ1.1} with $f(u)=|u|^pu$ defined on $[t_0,T]\times\R^2$ and satisfying
\begin{equation}\label{equ3.261}
\|u\|_{L_{t,x}^4([t_0,T]\times\R^2)}\leq\eta
\end{equation}
for some small constant $\eta.$ Assume $E(Iu(t_0))\leq1$. Then for
$s>\tfrac{1+s_c}2$, $p\geq2$ and sufficiently large $N,$ we have for
any $t\in[t_0,T]$
\begin{equation}\label{equ3.361}
Z_I(t)\lesssim \|\nabla Iu(t_0)\|_2+g_1(t)^pZ_I(t)+
N^{-(s-s_c)}Z_I(t)g_1(t)^{p-1}\big[g_1(t)+h_1(t)\big],
\end{equation}
where $g_1(t)$ and $h_1(t)$ are defined by
\begin{equation}\label{equ3.461}
g_1(t)^p=\eta^{2}\sup_{s\in[t_0,t]}E(Iu(s))^{\frac{2p-4}{p(p+2)}}+\eta^{\theta
p}Z_I(t)^{(1-\theta)p}+ N^{-(1-s_c)p}Z_I(t)^p
\end{equation}
with $\theta=\tfrac{2}{3p-4}$ and
\begin{equation}\label{equ3.4.161}
h_1(t)=\eta^{\frac4d}\sup_{s\in[t_0,t]}E(Iu(s))^{\frac{2(p-2)}{p(p+2)}}+Z_I(t).
\end{equation}
\end{proposition}

\begin{proof}
The proof is similar to the proof of Proposition \ref{zitinc}. The only
difference is that we use $\|u\|_{L_{t,x}^4}$ instead of
$\|u\|_{L_{t,x}^{5}}$ in estimates \eqref{equ3.1d},
\eqref{equ3.1d1} and \eqref{equbuch}. In this way,  one  can relax the restriction of $p$,
and to obtain estimate \eqref{equ3.361}
 for  $p\geq2$.

\end{proof}

Next, we show the energy increment of $E(Iu)(t)$.

\begin{proposition}[Energy increment]\label{enerin61}
Let $u(t,x)$ be an $H^s$ solution to problem \eqref{equ1.1} with
$f(u)=|u|^pu$ defined on $[t_0,T]\times\R^2$ and satisfying
\begin{equation}\label{equ3.261a}
\|u\|_{L_{t,x}^4([t_0,T]\times\R^2)}\leq\eta
\end{equation}
for some small constant $\eta.$ Assume $E(Iu(t_0))\leq1$. Then for
$s\geq\max\{\tfrac{p}{p+1},\tfrac{1+s_c}2\}$, $p\geq2$ and
sufficiently large $N,$ we have for any $t\in[t_0,T]$
\begin{align}\label{equ3.761}
&\big|\sup_{s\in[t_0,t]}E(Iu(s))-E(Iu(t_0))\big|\\\nonumber
\lesssim&N^{-(1-s_c)}Z_I(t)^3g_1(t)^{p-1}+N^{-(s-s_c)}Z_I(t)^2g_1(t)^{p-1}\big[g_1(t)+h_1(t)\big]\\\nonumber
&+ k_1(t)\Big\{N^{-1}g_1(t)^pZ_I(t)+\eta
N^{-(s-s_c)}g_1(t)^{p-1}\big( g_1(t)+h_1(t)\big)\Big\},
\end{align}
where $g_1(t),~h_1(t)$ are defined as in Proposition
\ref{zitinc6.1}, and $k_1(t)$ is defined to be
\begin{align*}
k_1(t)=~\eta\sup_{s\in[t_0,t]}E(Iu(s))^{\frac{p}{p+2}}+
Z_I(t)^{p+1}.\end{align*}

\end{proposition}

\begin{proof}
The proof is similar to the proof of  Proposition \ref{enerin}. In fact, we use
$\|u\|_{L_{t,x}^4}$ instead of $\|u\|_{L_{t,x}^{5}}$ in estimates
\eqref{equ3.8}, \eqref{equ3.9}, \eqref{equ3.11} and \eqref{equ3.12}.
However,   we estimate \eqref{equ3.13} in a different way as
follows. By the assumption $s>\tfrac{1+s_c}2$, we have
$\nu:=2s-s_c-1\in(0,s)$.
 We obtain, by the same argument as deriving \eqref{equ3.6.1d1},
 the following estimate
\begin{align}\nonumber
&\big\|(Iu)|u|^p-I(|u|^pu)\big\|_{L_{t,x}^\frac{4}{3}}\\
\nonumber
\lesssim&N^{-(s-s_c)}\|Iu\|_{L_{t,x}^{4}}\big\|\langle\nabla\rangle^{s-s_c}
|u|^p\big\|_{L_{t,x}^{2}}\\
\nonumber
\lesssim&N^{-(s-s_c)}\eta\|u\|_{L_{t,x}^{2p}}^{p-1}
\Big(\|u\|_{L_{t,x}^{2p}}+\big\||\nabla|^{
s-s_c}u\big\|_{L_{t,x}^{2p}}\Big)\\
\lesssim&\eta
N^{-(s-s_c)}g_1(t)^{p-1}\Big( g_1(t)+h_1(t)\Big).\label{equ3.1361}
\end{align}
\end{proof}

Combining the above two propositions, a standard bootstrap argument
and the same argument as in the proof of Proposition \ref{almost}, we can show that
the quantity $E(Iu)(t)$ is ``almost conserved" in the following sense.

\begin{proposition}[Almost conservation law]\label{almost61}
Let $u(t,x)$ be an $H^s$ solution to problem \eqref{equ1.1} with
$f(u)=|u|^pu$ defined on $[t_0,T]\times\R^2$  and satisfying
\begin{equation}\label{equ3.2.161}
\|u\|_{L_{t,x}^4([t_0,T]\times\R^2)}\leq\eta
\end{equation}
for some small constant $\eta.$ Assume $E(Iu(t_0))\leq1$. Then for
$s\geq\max\big\{\tfrac{1+s_c}2,~\tfrac{p}{p+1}\big\}$, $p\geq2$ and
sufficiently large $N,$ we have
\begin{align}\label{equ3.3.161}
E(Iu)(t)=E(Iu(t_0))+O(N^{s_c-s}).
\end{align}

\end{proposition}

Now we turn to prove Theorem \ref{theorem1.2}.

\vskip 0.2cm

\noindent{\bf The proof of Theorem \ref{theorem1.2}:} Assume $u$ is
a solution to problem \eqref{equ1.1} with $f(u)=|u|^pu$, then so is
\begin{equation}\label{equ4.361}
u^\la(t,x)=\la^{-\frac2p}u\big(\tfrac{t}{\la^2},\tfrac{x}{\la}\big).
\end{equation}
Choosing  a sufficiently large $\la$ depending on
$\|u_0\|_{H^s}$ and $N$ such that
\begin{equation}\label{equ4.461}
N^{1-s}\la^{s_c-s}\|u_0\|_{\dot H^s}\ll1\quad\text{and}\quad
\la^{-\frac2p+\frac{2}{p+2}}\|u_0\|_{H^s}\ll1,
\end{equation}
we get
\begin{equation}
E(Iu_0^\la)=\tfrac12\|\nabla
 Iu_0^\la\|_{L^2}^2+\tfrac1{p+2}\|Iu_0^\la\|_{L^{p+2}}^{p+2}\ll1.
\end{equation}

Next we claim that for any  arbitrary large  $T_0>0$, there exists
an absolute constant $C$ such that
\begin{equation}\label{equ4.561}
\|u^\la\|_{L_{t,x}^{4}([0,\la^2T_0]\times\R^2)}\leq
C\la^{s_c}(\la^2T_0)^\frac{1}{12}.
\end{equation}
 We prove
this claim by the standard bootstrap argument. Let us define
$$\Omega_1:=\big\{t\in[0,\la^2T_0]:~\|u^\la\|_{L_{t,x}^{4}([0,t]\times\R^2)}\leq C\la^{s_c}t^\frac{1}{12}\big\}.$$
We want to show $\Omega_1=[0,\la^2T_0]$. Let
$$\Omega_2:=\big\{t\in[0,\la^2T_0]:~\|u^\la\|_{L_{t,x}^{4}([0,t]\times\R^2)}\leq 2C\la^{s_c}t^\frac{1}{12}\big\}.$$
By the same argument as deriving Proposition \ref{promain}, it suffices  to
prove that  for any  $T\in\Omega_2$, we have $T\in\Omega_1$. Throughout the following proof,  all
spacetime norms will be computed on $[0,T]\times\R^2$.

\vskip0.12cm

Using the interaction Morawetz estimate and the mass conservation,
we get
\begin{align}
\|u^\la\|_{L_{t,x}^{4}}^{4}\lesssim&T^\frac13\|u_0^\la\|_2^2\|u^\la\|_{L_t^\infty
\dot H^\frac12}^2+T^\frac13\|u_0^\la\|_2^4\nonumber\\
\lesssim&T^\frac13\la^{{2s_c}}\|u_0\|_2^2\|u^\la\|_{L_t^\infty \dot
H^\frac12}^2+T^\frac13\la^{4s_c}\|u_0\|_2^4. \label{equ4.6d161}
\end{align}
From \eqref{equ4.7} and \eqref{equ4.8}, we have the control of  $\|u^\la(t)\|_{L_t^\infty\dot H^\frac12}$ as follows
\begin{align*}
\|u^\la(t)\|_{L_t^\infty\dot H^\frac12}\lesssim&\big\|P_{\leq
N}u^\la\big\|_{\dot H^\frac12}+\big\|P_{>N}u^\la\big\|_{\dot
H^\frac12}\\
\lesssim&\la^{\frac12s_c}\|u_0\|_2^\frac12\|Iu^\la\|_{\dot
H^1}^\frac12+\la^{s_c-\frac12}\|u_0\|_2^{1-\frac1{2s}}\|Iu^\la\|_{\dot
H^1}^\frac1{2s}.
\end{align*}
Plugging this into \eqref{equ4.6d161}, we estimate
\begin{align*}
\|u^\la\|_{L_{t,x}^{4}}^4\lesssim&_{\|u_0\|_2}T^\frac13\la^{{2s_c}}\|u^\la\|_{L_t^\infty
\dot
H^\frac12}^2+T^\frac13\la^{4s_c}\\
\lesssim&_{\|u_0\|_2}T^\frac13\la^{3s_c}\sup_{s\in[0,T]}\Big(\|Iu^\la\|_{\dot
H^1}+\la^{s_c-1}\|Iu^\la\|_{\dot
H^1}^\frac1{s}\Big)+T^\frac13\la^{4s_c}\\
\leq&C(\|u_0\|_2)T^\frac13\Big[\la^{3s_c}\sup_{s\in[0,T]}\big(\|Iu^\la\|_{\dot
H^1}+\|Iu^\la\|_{\dot H^1}^\frac1{s}\big)+\la^{4s_c}\Big],
\end{align*}
where we use the fact that $\la\gg1$ in the last inequality.
Thus, choosing $C$ sufficiently large depending on $\|u_0\|_2$, we
obtain $T\in\Omega_1$ provided we can deduce
\begin{equation}\label{equ4.1061}
\sup_{s\in[0,T]}\|Iu^\la\|_{\dot H^1}\leq1, \quad  T\in\Omega_2.
\end{equation}
 In fact, let $\eta>0$ be
sufficiently small constant as in Proposition \ref{almost61}, and we
divide $[0,T]$ into
\begin{equation}\label{equ4.11a}
L\sim\frac{T^\frac13\la^{4s_c}}{\eta^4}
\end{equation}
subintervals $I_j=[t_j,t_{j+1}]$ such that
\begin{equation}\label{equ4.1261}
\|u^\la\|_{L_{t,x}^{4}(I_j\times\R^2)}\leq\eta.
\end{equation}
Using Proposition \ref{almost61} on each interval $I_j,$ we
obtain
\begin{equation}\label{equ4.1361}
\sup_{t\in[0,T]}E(Iu^\la)(t)\leq E(Iu_0^\la)+LN^{s_c-s}.
\end{equation}
To control  small energy during the iteration, we need
$$LN^{s_c-s}\simeq T^\frac13\la^{4s_c}N^{s_c-s}\ll1
.$$
This property together with \eqref{equ4.461} and $T\leq\la^2T_0$ leads to
\begin{equation}\label{relat}
T_0^\frac13N^{(\frac23+4s_c)\frac{1-s}{s-s_c}-(s-s_c)}=T_0^\frac13N^{-\frac{3(s-s_c)^2-2(1+6s_c)(1-s)}{3(s-s_c)}}\ll1
\end{equation}
 by choosing $N=N(\|u_0\|_{H^s},T_0)$ large enough
provided that $s$ satisfies
$$3(s-s_c)^2-2(1+6s_c)(1-s)>0,$$ i.e. $s>\tilde s_1$, where $\tilde s_1$ is the positive root of the
quadratic equation
$$3(s-s_c)^2-2(1+6s_c)(1-s)=0.$$ This completes the bootstrap argument and hence
the claim \eqref{equ4.561}.

To estimate $\|u(t)\|_{H^s}$, by the conservation of mass, \eqref{equ2.3} and \eqref{equ4.1061}, we get for
$t\in[0,T_0]$
\begin{align*}
\|u(t)\|_{H^s}\lesssim&\|u_0\|_{L^2}+\|u(t)\|_{\dot
H^s}\\
\lesssim&\|u_0\|_{L^2}+\la^{s-s_c}\|u^\la(\la^2t)\|_{\dot
H^s}\\
\lesssim&\|u_0\|_{L^2}+\la^{s-s_c}\|Iu^\la(\la^2t)\|_{H^1}\\
\lesssim&\|u_0\|_{L^2}+\la^{s-s_c}\big(\|u_0^\la\|_2+\|Iu^\la(\la^2t)\|_{\dot
H^1}\big)\\
\lesssim&\|u_0\|_{L^2}+\la^{s-s_c}\big(\la^{s_c}\|u_0\|_2+1\big)\\
\leq&C(\|u_0\|_{H^s})(1+\la^s)\\
\leq&C(\|u_0\|_{H^s})(1+T_0)^{\frac{1-s}{3(s-s_c)^2-2(1+6s_c)(1-s)}+},
\end{align*}
where we use the relationship \eqref{equ4.461} and \eqref{relat} in
the last inequality. This completes the proof of Theorem
\ref{theorem1.2}.




\section{Proof of Theorem \ref{thm1.2}}

In this section, we consider the Cauchy problem for the nonlinear Schr\"odinger equation
\begin{align} \label{equ1.1.1}
\begin{cases}    (i\partial_t+\Delta)u= f(u)=|u|^pu+|u|^{2k}u,\quad
(t,x)\in\R\times\R^2,
\\
u(0,x)=u_0(x)\in H^s(\R^2).
\end{cases}
\end{align}
If $u(t,x)$ is the solution to \eqref{equ1.1.1}, then
$$u^\la(t,x)=\la^{-\frac1k}u\big(\tfrac{t}{\la^2},
\tfrac{x}{\la}\big)$$ is the solution to
\begin{align} \label{equ1.1.2}
\begin{cases}
(i\partial_t+\Delta)u^\la=\la^{-2+\frac{p}k}|u^\la|^pu^\la+|u^\la|^{2k}u^\la
\\
u^\la(0,x)=\la^{-\frac1k}u_0\big(\frac{x}\la\big).
\end{cases}
\end{align}
The energy $E(u^\la)$ is defined by
\begin{equation*}\label{equ5.1}
E(u^\la)(t)=\frac12\int|\nabla
u^\la(t,x)|^2dx+\frac1{2(k+1)}\int|u^\la(t,x)|^{2(k+1)}dx+\frac{\la^{-2+\frac{p}k}}{p+2}\int|u^\la(t,x)|^{p+2}dx.
\end{equation*}
For given $u_0\in H^s(\R^2)$, we have
\begin{align*}
\|\nabla Iu_0^\la\|_2\leq N^{1-s}\|u_0^\la\|_{\dot
H^s}=N^{1-s}\la^{1-\frac1k-s}\|u_0\|_{\dot H^s},\\
\la^{-2+\frac{p}k}\|u_0^\la\|_{p+2}^{p+2}=\la^{-2+\frac{p}k}\la^{-\frac{p+2}k+2}\|u_0\|_{p+2}^{p+2}\lesssim\la^{-\frac{2}k
}\|u_0\|_{H^s_x}^{p+2},\\
\|u_0^\la\|_{2k+2}=\la^{-\frac1k+\frac{1}{k+1}}\|u_0\|_{2k+2}\lesssim\la^{-\frac1k+\frac{1}{k+1}}\|u_0\|_{H^s_x}.
\end{align*}
As $s>1-\tfrac1k$, choosing $\la$ sufficiently large depending on
$\|u_0\|_{H^s}$ and $N$ such that
\begin{equation}\label{equ5.3.1.1} N^{1-s}\la^{1-\frac1k-s}\|u_0\|_{\dot
H^s}\ll1~~\text{and}~~\la^{-\frac1k+\frac{1}{k+1}}\|u_0\|_{H^s_x}\ll1,
\end{equation} we obtain
\begin{equation}\label{equ5.41}
E(Iu_0^\la)\leq1.
\end{equation}

\subsection{Almost conservation law}
 Let us define $Z_I(t)$ by
 \begin{equation}\label{equ3.1.12}
 Z_I(t):=\|Iu\|_{Z(t)}=\sup_{(q,r)\in\Lambda_0}\Big(\sum_{N\geq1}\|\nabla
 P_NIu(t)\|_{L_t^qL_x^r([t_0,t)\times\R^2)}^2\Big)^\frac12.
 \end{equation}
Moreover, we denote
$$s_c^{(1)}=1-\tfrac2p,~~s_c^{(2)}=1-\tfrac1k,~~s_c^{(1)}<s_c^{(2)}.$$

\begin{proposition}\label{zitinc5.1}
Let $u(t,x)$ be an $H^s$ solution to problem \eqref{equ1.1.2} defined on
$[t_0,T]\times\R^2$ and satisfying
\begin{equation}\label{equ5.2}
\|u\|_{L_{t,x}^{5}([t_0,T]\times\R^2)}\leq\eta
\end{equation}
for some small constant $\eta.$ Assume $E(Iu(t_0))\leq1$. Then for
sufficiently large $N,$
$$s>\max\big\{s_c^{(1)},\tfrac{1+s_c^{(1)}}2,s_c^{(2)}\big\}=\max\big\{\tfrac{1+s_c^{(1)}}2,s_c^{(2)}\big\},
~2k>p\geq\tfrac{5}{2},$$ and $k$ is an integer number larger than one,
\begin{align}Z_I(t)\lesssim \|\nabla Iu_0\|_2+\la^{-2+\frac{p}k}\Big[\tilde
g(t)^pZ_I(t)+ N^{-(s-s_c)}Z_I(t)\tilde g(t)^{p-1}\big(\tilde
g(t)+\tilde h(t)\big)\Big]+g_k(t)Z_I(t). \label{equ5.3}
\end{align}
where $\tilde g(t)$, $\tilde h(t)$ and $g_k(t)$ are defined by
\begin{align*}
\begin{cases}
 \tilde
g(t)^p=&\eta^{\theta_1p}\la^{(2-\frac{p}k)\frac{(1-\theta_1)p}{p+2}}\sup\limits_{s\in[t_0,t]}
E(Iu(s))^{\frac{(1-\theta_1)p}{p+2}}+\eta^{\theta_2p}Z_I(t)^{(1-\theta_2)p}+
N^{-(1-s_c^{(1)})p}Z_I(t)^p\\
\tilde
h(t)=&\eta^{\theta_1}\la^{(2-\frac{p}k)\frac{1-\theta_1}{p+2}}\sup\limits_{s\in[t_0,t]}E(Iu(s))^{\frac{1-\theta_1}{p+2}}+Z_I(t)\\
g_k(t)=&\eta^{2k\theta_1}
\sup\limits_{s\in[t_0,t]}E(Iu(s))^\frac{(1-\theta_1)k}{k+1}+\eta^{2k\theta_2}Z_I(t)^{2k(1-\theta_2)}
+N^{-2k(1-s_c^{(2)})}Z_I(t)^{2k}\end{cases}
\end{align*}
with $\theta_1,~\theta_2\in(0,1)$ defined as in Proposition
\ref{zitinc}.
\end{proposition}

\begin{proof}
Using the Strichartz estimate \eqref{strest},  we get from \eqref{equ1.1.2}
\begin{align}\nonumber
Z_I(t)\lesssim&\|\nabla Iu(t_0)\|_2+\|\nabla
IF(u)\|_{L_{t,x}^\frac{4}{3}}\\
\label{equ5.4.1}\lesssim&\|\nabla
Iu(t_0)\|_2+\la^{-2+\frac{p}k}\|\nabla
I(|u|^pu)\|_{L_{t,x}^\frac{4}{3}}+\|\nabla
I(|u|^{2k}u)\|_{L_{t,x}^\frac{4}{3}},
\end{align}
where all space-time norms are computed on $[t_0,t)\times\R^2$.  Using the following estimate of low frequency part
\begin{align}\nonumber
\|u_{\leq1}\|_{L_{t,x}^{2p}}^p\lesssim&\|u_{\leq1}\|_{L_{t,x}^{5}}^{\theta_1p}\|u_{\leq1}\|_{L_{t,x}^\infty}^{(1-\theta_1)p}\lesssim
\eta^{\theta_1p}\|u_{\leq1}\|_{L_t^\infty
L_x^{p+2}}^{(1-\theta_1)p}\\\label{equ3.1d49}
\lesssim&\eta^{\theta_1p}\la^{(2-\frac{p}k)\frac{(1-\theta_1)p}{p+2}}\sup_{s\in[t_0,t]}E(Iu(s))^{\frac{(1-\theta_1)p}{p+2}},
\end{align}
we  obtain in the same way as
deriving \eqref{equ3.5} and \eqref{equ3.6.1d1}, the inequality
\begin{align}\label{equ3.31d}
\|\nabla I(|u|^pu)\|_{L_{t,x}^\frac{4}{3}} \lesssim\tilde
g(t)^pZ_I(t)+ N^{-(s-s_c)}Z_I(t)\tilde g(t)^{p-1}\big(\tilde
g(t)+\tilde h(t)\big),
\end{align}
where $\theta_1=\tfrac{5}{2p}$  with  $p\geq\tfrac{5}{2}$.

On the other hand, using the fact that $\nabla I$ acts as a
derivative, we obtain
\begin{align}\label{equ5.4.2}
\|\nabla
I(|u|^{2k}u)\|_{L_{t,x}^\frac{4}{3}}\lesssim&\|u\|_{L_{t,x}^{4k}}^{2k}\|\nabla
Iu\|_{L_{t,x}^4}, \quad \forall \; t\in[t_0,T].
\end{align}
Using estimate \eqref{gujiagj} with $p=2k$, we deduce
\begin{equation*}
\|u\|_{L_{t,x}^{4k}}^{2k}\lesssim\eta^{2k\theta_1}
\sup_{s\in[t_0,t]}E(Iu(s))^\frac{(1-\theta_1)k}{k+1}+\eta^{2k\theta_2}Z_I(t)^{2k(1-\theta_2)}
+N^{-2}Z_I(t)^{2k}.
\end{equation*}
This together with \eqref{equ5.4.1}, \eqref{equ3.31d} and
\eqref{equ5.4.2} yields \eqref{equ5.3}.

\end{proof}

\begin{proposition}[Energy increment]\label{enerin5.1}
Let $u(t,x)$ be an $H^s$ solution to problem \eqref{equ1.1.2} defined on
$[t_0,T]\times\R^2$ and satisfying
\begin{equation}\label{equ5.6}
\|u\|_{L_{t,x}^{5}([0,T]\times\R^2)}\leq\eta
\end{equation}
for some small constant $\eta.$ Assume $E(Iu(t_0))\leq1$. Then for
$$s\geq\tfrac{2k}{2k+1},~\quad ~2k>p\geq\tfrac{11}{4},~\quad\; 2\leq k\in
\mathbb{N}$$ and sufficiently large $N,$ we have
\begin{align}\label{equ5.7}
\big|\sup_{s\in[t_0,t]}E(Iu(s))-E(Iu(t_0))\big| \lesssim~
h_1(t)+h_2(t)+h_3(t)+h_4(t),
\end{align}
where the quantities $h_j(t)~(j=1,2,3, 4)$ are  defined below in \eqref{equ5.6.1}, \eqref{equ5.6.2}, \eqref{equ5.6.3},\eqref{equ5.6.4}.

\end{proposition}

\begin{proof}
Since
 $$iIu_t+\Delta Iu=If(u)=\la^{-2+\frac{p}k}I\big(|u|^pu\big)+I\big(|u|^{2k}u\big),$$
by a simple computation, we obtain
\begin{align}\nonumber
E(Iu(t))-E(Iu(t_0)) =&\int_{t_0}^t\frac{\pa}{\pa
s}E(Iu(s))ds\\\label{equ5.5.1} =&{\rm
Im}\int_{t_0}^t\int_{\R^d}\overline{\Delta
Iu}\Big[|Iu|^{2k}Iu-I(|u|^{2k}u)\Big]dxds\\\label{equ5.5.2} &-{\rm
Im}\int_{t_0}^t\int_{\R^d}\overline{I(|u|^{2k}u)}\Big[|Iu|^{2k}Iu-I(|u|^{2k}u)\Big]dxds\\\label{equ5.5.3}
&-\la^{-2+\frac{p}k}{\rm Im}\int_{t_0}^t\int_{\R^d}\overline{\nabla
Iu}\cdot\nabla \Big[|Iu|^pIu-I(|u|^pu)\Big]dxds\\\label{equ5.5.4}
&-\la^{-4+\frac{2p}k}{\rm
Im}\int_{t_0}^t\int_{\R^d}\overline{I(|u|^pu)}\Big[|Iu|^pIu-I(|u|^pu)\Big]dxds\\\label{equ5.5.5}
&-\la^{-2+\frac{p}k}{\rm
Im}\int_{t_0}^t\int_{\R^d}\overline{I(|u|^pu)}\Big[|Iu|^{2k}Iu-I(|u|^{2k}u)\Big]dxds\\\label{equ5.5.6}
&-\la^{-2+\frac{p}k}{\rm
Im}\int_{t_0}^t\int_{\R^d}\overline{I(|u|^{2k}u)}\Big[|Iu|^pIu-I(|u|^pu)\Big]dxds.
\end{align}
Recalling the result in [\cite{CGT09}, Proposition 5.2], we have for
$s>\tfrac{2(k-1)}{2k-1}$
\begin{align}\nonumber
&\big|\eqref{equ5.5.1}+\eqref{equ5.5.2}\big|\\\nonumber
\lesssim&~N^{-1+}\Big(Z_I(t)^{2k+2}+\eta^{2k\theta_5}Z_I(t)^2\sup_{s\in[t_0,t]}E(Iu(s))^\frac{(1-\theta_5)k}{k+1}\\\nonumber
&\qquad\quad\quad+\sum_{J=3}^{2k+2}
\eta^{(2k+2-J)\theta_6}Z_I(t)^J\sup_{s\in[t_0,
t]}E(Iu(s))^{\frac{(1-\theta_6)(2k+2-J)}{2(k+1)}}\Big)\\\nonumber
&+N^{-1+}\Big(Z_I(t)^{2k+1}+\eta^{2k\theta_5}Z_I(t)\sup_{s\in[t_0,t]}E(Iu(s))^\frac{(1-\theta_5)k}{k+1}\Big)\\\nonumber
&\qquad\times\Big(Z_I(t)^{2k+1}+\eta^{\theta_7(2k+1)}
\sup_{s\in[t_0,t]}
E(Iu(s))^\frac{(1-\theta_7)(2k+1)}{2(k+1)}\Big)\\\nonumber
&+N^{-1+}\sum_{J=3}^{2k+2}\eta^{(2k+2-J)\theta_6}Z_I(t)^{J-1}\sup_{s\in[t_0,
t]}E(Iu(s))^{\frac{(1-\theta_6)(2k+2-J)}{2(k+1)}}\\\nonumber
&\qquad\times\Big(Z_I(t)^{2k+1}+\eta^{\theta_7(2k+1)}
\sup_{s\in[t_0,t]}
E(Iu(s))^\frac{(1-\theta_7)(2k+1)}{2(k+1)}\Big)\\\label{equ5.6.1}
=:&h_1(t),
\end{align}
where $\theta_5,\theta_6,\theta_7\in(0,1)$ are defined by
$$\theta_5=\tfrac{5}{4k},~\theta_6=\tfrac{5}{4(2k-1)},~\theta_7=\tfrac{5}{3(2k+1)}.$$
Here, we adopt the interaction Morawetz norm $L_{t,x}^5$ instead of
$L_t^4L_x^8$-norm as used in \cite{CGT09}. There is only one
difference appearing in the power of $\eta$ and $E(Iu)$.

While, by Proposition \ref{enerin}, we have for
$s\geq\tfrac{p}{p+1}$ and $p\geq\tfrac{11}{4}$
\begin{align}
&\big|\eqref{equ5.5.3}+\eqref{equ5.5.4}\big|\nonumber\\
\lesssim&\la^{-2+\frac{p}k}\Big(N^{-(1-s_c^{(1)})}Z_I(t)^3\tilde
g(t)^{p-1}+N^{-(s-s_c^{(1)})}Z_I(t)^2\tilde g(t)^{p-1}\big[\tilde
g(t)+\tilde h(t)\big]\Big)\nonumber \\
&+\la^{-4+\frac{2p}k}\tilde
k(t)\Big(N^{-1}\tilde g(t)^pZ_I(t)+\eta
N^{-(s-s_c^{(1)}+\frac1q)}\tilde m(t)\Big)=: h_2(t), \label{equ5.6.2}
\end{align}
where $q=\frac{5p}{11}$, and  $\tilde k(t),~\tilde m(t)$ are defined
by
\begin{equation*}
\begin{cases}
\tilde
k(t)=\eta^{\theta_3(p+1)}\la^{(2-\frac{p}k)\frac{(1-\theta_3)(p+1)}{p+2}}\sup\limits_{s\in[t_0,t]}E(Iu(s))^{\frac{(1-\theta_3)(p+1)}{p+2}}+
Z_I(t)^{p+1},~\theta_3=\tfrac{5}{3(p+1)}\\
\tilde m(t)=\eta^{1+\theta_4}\tilde
g(t)^{(1-\theta_4)(p-1)}\Big(\eta^{\theta_4} \tilde
g(t)^{1-\theta_4}+Z_I(t)\Big),~\theta_4=\tfrac{1}{4p-10}.
\end{cases}
\end{equation*}

{\bf $\bullet$ The estimate of \eqref{equ5.5.5}:} By the same
argument as leading to \eqref{equ5.5.1}, we have for
\begin{align}\nonumber
|\eqref{equ5.5.5}|\lesssim&\la^{-2+\frac{p}k}N^{-1+}\sup_N\|P_NI(|u|^pu)\|_{L_{t,x}^4}
\Big\{Z_I(t)^{2k+2}+\eta^{2k\theta_5}Z_I(t)^2\sup_{s\in[t_0,t]}E(Iu(s))^\frac{(1-\theta_5)k}{k+1}\\\label{eque5.5}
&+\sum_{J=3}^{2k+2}
\eta^{(2k+2-J)\theta_6}Z_I(t)^J\sup_{s\in[t_0,
t]}E(Iu(s))^{\frac{(1-\theta_6)(2k+2-J)}{2(k+1)}}\Big\},\;\;  s>\frac{2(k-1)}{2k-1}.
\end{align}
To estimate $\|P_NI(|u|^pu)\|_{L_{t,x}^4}$,   we
 obtain by \eqref{equ3.11}
\begin{align*}
\|P_NI(|u|^pu)\|_{L_{t,x}^4}\lesssim&\|u\|_{L_{t,x}^{4(p+1)}}^{p+1}\\
\lesssim&\eta^{\theta_3(p+1)}\la^{(2-\frac{p}k)\frac{(1-\theta_3)(p+1)}{p+2}}\sup_{s\in[t_0,t]}E(Iu(s))^{\frac{(1-\theta_3)(p+1)}{p+2}}+
Z_I(t)^{p+1},
\end{align*}
where $\theta_3=\tfrac{5}{3(p+1)}$ and we need the restriction
$s\geq\frac{p}{p+1}$ in the estimate of the high frequency part by
means of \eqref{equ2.2} with $\sigma=\frac{p}{p+1}$. Plugging this
into \eqref{eque5.5} gives
\begin{align}\nonumber
&|\eqref{equ5.5.5}|\\\nonumber
\lesssim&\la^{-2+\frac{p}k}N^{-1+}
\Big(\eta^{\theta_3(p+1)}\la^{(2-\frac{p}k)\frac{(1-\theta_3)(p+1)}{p+2}}\sup_{s\in[t_0,t]}E(Iu(s))^{\frac{(1-\theta_3)(p+1)}{p+2}}+
Z_I(t)^{p+1}\Big)\\\nonumber
&\times\Big\{Z_I(t)^{2k+2}+\eta^{2k\theta_5}Z_I(t)^2\sup_{s\in[t_0,t]}E(Iu(s))^\frac{(1-\theta_5)k}{k+1}\\\nonumber
&\qquad\quad+\sum_{J=3}^{2k+2}
\eta^{(2k+2-J)\theta_6}Z_I(t)^J\sup_{s\in[0,
t]}E(Iu(s))^{(1-\theta_6)(2k+2-J)}\Big\}\\\label{equ5.6.3} =:&h_3(t).
\end{align}

{\bf $\bullet$ The estimate of \eqref{equ5.5.6}:} By the same
argument as deducing \eqref{equ3.10}, we get
\begin{align}\nonumber
&|\eqref{equ5.5.6}|\\\nonumber
\lesssim&\la^{-2+\frac{p}k}\|u\|_{L_{t,x}^{4(2k+1)}}^{2k+1}
\Big[\big\|Iu(|Iu|^p-|u|^p)\big\|_{L_{t,x}^\frac{4}{3}}+\big\|(Iu)|u|^p-I(|u|^pu)\big\|_{L_{t,x}^\frac{4}{3}}\Big]\\\nonumber
\lesssim&\la^{-2+\frac{p}k}\bigg\{\eta^{\theta_3(2k+1)}\sup_{s\in[t_0,t]}E(Iu(s))^{\frac{(1-\theta_3)(2k+1)}{2(k+1)}}+
Z_I(t)^{2k+1}\bigg\}\\\nonumber &\times \Big\{N^{-1}\tilde
g(t)^pZ_I(t)+ N^{-(s-\frac12+\frac1q)}\tilde m(t)\Big\}\\\label{equ5.6.4}
 \simeq&h_4(t),
\end{align}
where  $q=\tfrac{5p}{11}$, $\tilde g(t)$ is defined as in
Proposition \ref{zitinc5.1}. We need the restriction $s\geq\tfrac{2k}{2k+1}$ in
the estimate of the high frequency part by means of \eqref{equ2.2}
with $\sigma=\frac{2k}{2k+1}$.

Note that
$\tfrac{dk}{2k+1}>\max\{\tfrac{2(k-1)}{2k-1},\tfrac{p}{p+1}\}$ and
collecting \eqref{equ5.6.1}-\eqref{equ5.6.4}, we obtain
\eqref{equ5.7}. Therefore, we conclude the proof of this
proposition.
\end{proof}

\vskip0.12cm

Combining the above two propositions, a standard bootstrap argument
and arguing as in the proof of Proposition \ref{almost}
 we can
easily show that the quantity $E(Iu)(t)$ is ``almost conserved" by using the condition
$$s_c^{(2)}<\max\big\{\tfrac{1+s_c^{(1)}}2,~\tfrac{dk}{2k+1}\big\}.$$

\begin{proposition}[Almost conservation law]\label{almost5.1}
Let $u(t,x)$ be an $H^s$ solution to problem \eqref{equ1.1.2} with  defined
on $[0,T]\times\R^2$ and satisfying
\begin{equation}\label{equ5.8}
\|u\|_{L_{t,x}^{5}([0,T]\times\R^2)}\leq\eta
\end{equation}
for some small constant $\eta.$ Assume $E(Iu_0)\leq1$. Then for
$$s\geq\max\big\{\tfrac{1+s_c^{(1)}}2,\tfrac{2k}{2k+1}\big\},
~~2k>p\geq\tfrac{11}{4},~2\leq k\in\mathbb{N}$$
 and sufficiently large $N,$ we
have
\begin{align}\label{equ5.9}
E(Iu)(t)=E(Iu_0)+O\big(\max\big\{N^{-1+},\la^{-(2-\frac{p}k)(1-\frac{(1-\theta_1)p}{p+2})}N^{-(s-s_c^{(1)})}\big\}\big),
\end{align}
where $\theta_1=\tfrac{5}{2p}$.
\end{proposition}

\begin{proof}
By the same way as deducing Proposition \ref{almost}, we derive that
the contributions of $h_j(t)(j=1,2,3,4)$ to the difference
$E(Iu(t))-E(Iu(t_0))$ are
\begin{align*}
&N^{-1},~\la^{-(2-\frac{p}k)(1-\frac{(1-\theta_1)p}{p+2})}N^{-(s-s_c^{(1)})},~\la^{-(2-\frac{p}k)}N^{-1+},\\
&\la^{-(2-\frac{p}k)}
N^{-(s-\frac12+\frac1q)}\la^{(2-\frac{p}k)(1-\theta_1)(1-\theta_4)\frac{p}{p+2}}
\end{align*}
respectively. This fact gives the formula \eqref{equ5.9}.
\end{proof}

\subsection{Global well-posedness and scattering}  By an
argument as similar to that in Section 3, we can reduce the proof of Theorem
\ref{thm1.2} to the following proposition.

\begin{proposition}\label{promain5.1}
Suppose $u(t,x)$ is a global solution to problem \eqref{equ1.1} with
$f(u)=|u|^pu+|u|^{2k}u$ and $u_0\in C_0^\infty(\R^2)$. Then
for
$$2k>p\geq\tfrac{11}{4},~1<
k\in\mathbb{N},$$ and
$$s\in(\tilde s_3,1),~\tilde s_3:=\max\big\{\tfrac{1+s_c^{(1)}}2,\tfrac{2k}{2k+1},\tfrac{5s_c^{(2)}}{4s_c^{(2)}+1},~s_3\big\}$$
with $s_3$ being the positive root of the quadratic equation
$$s^2-(s_c^{(1)}+s_c^{(2)}-\al)s-\al=0,~\al=4s_c^{(2)}-\tfrac{9(2-\frac{p}k)}{2(p+2)},$$ we have
\begin{align}\label{equ5.10}
\|u\|_{L_{t,x}^5(\R\times\R^2)}\leq&
C\big(\|u_0\|_{H^s(\R^2)}\big),\\\label{equ5.11}
\sup_{t\in\R}\|u(t)\|_{H^s(\R^2)}\leq&C\big(\|u_0\|_{H^s(\R^2)}\big).
\end{align}
\end{proposition}

\begin{proof}
By the same argument as in the proof of Proposition \ref{promain} and the  scaling transform,
we claim that
\begin{equation}\label{equ5.2.1}
\|u^\la\|_{L_{t,x}^5(\R\times\R^2)}\leq C\la^{\frac{4}{5}s_c^{(2)}}.
\end{equation}
Indeed, we define
$$\Omega_1:=\big\{t\in[0,\infty):~\|u^\la\|_{L_{t,x}^5([0,t]\times\R^2)}\leq C\la^{\frac{4}{5}s_c^{(2)}}\big\}.$$
We want to show $\Omega_1=[0,\infty)$. Let
$$\Omega_2:=\big\{t\in[0,\infty):~\|u^\la\|_{L_{t,x}^5([0,t]\times\R^2)}\leq 2C\la^{\frac{4}{5}s_c^{(2)}}\big\}.$$
Then, it suffices to show $\Omega_2\subset\Omega_1$ by the standard
bootstrap argument. In the same way as
deriving \eqref{equ4.5}, it is sufficient to prove
\begin{equation}\label{equ5.4.10}
\sup_{s\in[0,T]}\|Iu^\la\|_{\dot H^1}\leq1, \quad  T\in\Omega_2.
\end{equation}
 In fact, let $\eta>0$ be a
sufficiently small constant as in Proposition \ref{almost}, and we
divide $[0,T]$ into
\begin{equation}\label{equ5.4.11}
L\sim\Big(\frac{\la^{\frac{4}{5}s_c^{(2)}}}\eta\Big)^5\sim\la^{4s_c^{(2)}},
\end{equation}
subintervals $I_j=[t_j,t_{j+1}]$ such that
\begin{equation}\label{equ5.4.12}
\|u^\la\|_{L_{t,x}^5(I_j\times\R^2)}\leq\eta.
\end{equation}
Using Proposition \ref{almost5.1} on each interval $I_j,$ we
obtain
\begin{equation}\label{equ5.4.13}
\sup_{t\in[0,T]}E(Iu^\la)(t)\leq
E(Iu_0^\la)+L{\mathcal O}\big(\max\big\{N^{-1+},~\la^{-(2-\frac{p}k)(1-\frac{(1-\theta_1)p}{p+2})}N^{-(s-s_c^{(1)})}\big\}\big).
\end{equation}
To maintain a small energy during the iteration, we need the estimate
\begin{align*}
&L{\mathcal O}\big(\max\big\{N^{-1+},~\la^{-(2-\frac{p}k)(1-\frac{(1-\theta_1)p}{p+2})}N^{-(s-s_c^{(1)})}\big\}\big)\\
\simeq&\la^{4s_c^{(2)}}
{\mathcal O}\big(\max\big\{N^{-1+},~\la^{-(2-\frac{p}k)(1-\frac{(1-\theta_1)p}{p+2})}N^{-(s-s_c^{(1)})}\big\}\big)\ll1,
\end{align*}
which together with \eqref{equ5.3.1.1} leads to
$$\la^{4s_c^{(2)}}N^{-1+}\simeq
N^{4s_c^{(2)}\frac{1-s}{s-s_c^{(2)}}}N^{-1+}\simeq
N^{\frac{(3p-4-\frac4p)-s(3p-4)}{p(s+\frac2p-1)-}}\ll1$$ and
$$\la^{4s_c^{(2)}}\la^{-(2-\frac{p}k)(1-\frac{(1-\theta_1)p}{p+2})}N^{-(s-s_c^{(1)})}\simeq N^{\frac{(1-s)\al}{s-s_c^{(2)}}-(s-s_c^{(1)})}\ll1$$
with
$$\al=4s_c^{(2)}-(2-\tfrac{p}k)(1-\tfrac{(1-\theta_1)p}{p+2})=4s_c^{(2)}-\tfrac{9(2-\frac{p}k)}{2(p+2)}.$$
 They may be
ensured by choosing $N=N(\|u_0\|_{H^s})$ large enough provided
$$s>\max\{\tfrac{5s_c^{(2)}}{4s_c^{(2)}+1},s_3\},$$
where $s_3$ is the positive root of the quadratic equation
$$s^2-(s_c^{(1)}+s_c^{(2)}-\al)s-\al=0.$$ This completes the bootstrap argument,  and hence
 Proposition \ref{promain5.1} follows. Therefore, we
conclude Theorem \ref{thm1.2}.
\end{proof}

\appendix

\section*{Appendix}
\setcounter{equation}{0}
\renewcommand{\theequation}{A.\arabic{equation}}
\renewcommand{\thetheorem}{A.\arabic{theorem}}
 In this Appendix, we state the result in the one
dimension. In fact, the proof is the same as in the case dimension two. We
utilize the following classical interaction Morawetz estimates in
\cite{CHVZ,PV}
\begin{equation}
\|u\|_{L_{t,x}^8(I\times\R)}\lesssim\|u\|_{L_t^\infty(I;\dot
H^\frac12(\R))}^\frac{1}{4}\|u_0\|_2^\frac3{4},
\end{equation}
and the improved interaction Morawetz estimates in \cite{CGT07}
\begin{equation}
\int_0^T\int_{\R}|u(t,x)|^6dxdt\lesssim
T^\frac13\|u_0\|_{L_x^2}^4\|u\|_{L_t^\infty([0,T],\dot
H^{1/2}_x)}^2+T^\frac13\|u_0\|_{L_x^2}^6,
\end{equation}
instead of \eqref{equ2.14rest} and \eqref{improved}.

Define
\begin{align}
s_0:=&\max\big\{\tfrac{1+s_c}2,~\tfrac{p}{2(p+1)},~ s_1\big\},
s_c=\tfrac{1}2-\tfrac2p,\\
\tilde s_0:=&\max\big\{\tfrac{1+s_c}2,~\tfrac{dp}{2(p+1)},~ \tilde
s_1\big\},
\end{align}
where $s_1$ is the positive root of the quadratic equation
$$s^2+5s_cs+s_c^2-7s_c=0,$$
and $\tilde s_1$ is the positive root of the quadratic equation
$$3(s-s_c)^2-2(1+9s_c)(1-s)=0.$$

\begin{theorem}\label{theorem1d}
{\rm (i)}\;  Assume that $u_0\in H^s(\R)$ with $s\in(s_0,1)$ and
$p\geq\tfrac{17}{3}$. Then the solution $u$ to $iu_t+\Delta
u=|u|^pu$ is global and scatters.

{\rm (ii)}\;Assume that $u_0\in H^s(\R)$ with $s\in(\tilde
s_0,1)$ and $p\geq4$. Then the solution $u$ to  $iu_t+\Delta
u=|u|^pu$ is global. Furthermore,  for any $T>0$,
\begin{equation}
\sup_{t\in[0,T]}\|u(t)\|_{H^s(\R)}\leq
C\big(\|u_0\|_{H^s(\R)}\big)(1+T)^{\frac{1-s}{3(s-s_c)^2-2(1+9s_c)(1-s)}+}.
\end{equation}

{\rm (iii)}\; Assume that $u_0\in H^s(\R)$ with
\begin{equation*}
s\in\Big(\max\big\{\tfrac{1+s_c^{(2)}}2,~\tfrac{p_2}{2(p_2+1)},~
s_2\big\},~1\Big),\;
s_c^{(j)}=\tfrac{1}2-\tfrac2{p_j},~j=1,2,~\tfrac{17}{3}\leq p_1<p_2
\end{equation*}
and $s_2$ is the positive root of the quadratic equation
$$s^2+5s_c^{(2)}s+(s_c^{(2)})^2-7s_c^{(2)}=0.$$
 Then the solution $u$ to
$iu_t+\Delta u=|u|^{p_1}u+|u|^{p_2}u$ is global and scatters.
\end{theorem}

\vskip 0.5cm

\textbf{Acknowledgements}  The authors would like to thank
 Piotr Biler,  Benjamin Dodson, Grzegorz Karch and
 the anonymous referee
 for    comments and suggestions.  This work is
supported in part by the National Natural Science
 Foundation of China under grant No. 11171033, No. 11231006, and No. 11371059.

\begin{center}

\end{center}

\end{document}